\documentclass[a4paper,11pt]{amsart}
\usepackage{amssymb}
\usepackage{amsmath}
\usepackage{amsthm}
\usepackage[dvipsnames]{xcolor}
\usepackage{braket,comment}
\usepackage{hyperref}
\usepackage{pdfpages,faktor}
\usepackage{graphicx,mathabx,bbm}
\usepackage[shortlabels]{enumitem}
\usepackage{caption}
\usepackage{subcaption}
\usepackage{mathrsfs} 
\usepackage{tikz-cd} 
\usepackage{bm}
\usepackage{enumitem}
\usepackage{mathtools}
\usepackage{pgfplots}
\usepackage{import}
\usepackage{xifthen}
\usepackage{transparent}

\pgfplotsset{compat=1.18} 

\numberwithin{equation}{section}
\theoremstyle{plain}
\newtheorem{theorem}{Theorem}[section]

\newtheorem{prop}[theorem]{Proposition}

\newtheorem{question}[theorem]{Question}
\newtheorem*{question*}{Question}

\theoremstyle{definition}
\newtheorem{defn}[theorem]{Definition}

\newtheorem{remark}{Remark}

\newcommand{\natls}{{\mathbb N}}

\newcommand{\C}{\mathbb{C}}

\newcommand{\Z}{\mathbb{Z}}

\newcommand{\D}{\mathbb{D}}

\newcommand{\cD}{\mathcal{D}}
\newcommand{\cH}{\mathcal{H}}
\newcommand{\cT}{\mathcal{T}}

\newcommand{\disk}{\mathbb{D}}

\DeclareMathOperator{\PSL}{PSL}

\DeclareMathOperator{\Int}{Int}

\numberwithin{figure}{section}

\makeatletter
\DeclareFontFamily{U}{tipa}{}
\DeclareFontShape{U}{tipa}{m}{n}{<->tipa10}{}
\newcommand{\arc@char}{{\usefont{U}{tipa}{m}{n}\symbol{62}}}

\newcommand{\arc}[1]{\mathpalette\arc@arc{#1}}

\newcommand{\arc@arc}[2]{%
  \sbox0{$\m@th#1#2$}%
  \vbox{
    \hbox{\resizebox{\wd0}{\height}{\arc@char}}
    \nointerlineskip
    \box0
  }%
}
\makeatother

\makeatletter
\def\leftrightarrowsfill@{\arrowfill@\leftrarrows\Rrelbar\lrightarrows}
\newcommand{\xleftrightarrows}[2][]{\ext@arrow 3399\leftrightarrowsfill@{#1}{#2}}
\makeatother

\author[M.~Mj]{Mahan Mj}
\address{School of Mathematics, Tata Institute of Fundamental Research, 1 Homi Bhabha Road, Mumbai 400005, India}
\email{mahan@math.tifr.res.in, mahan.mj@gmail.com}

\author[S.~Mukherjee]{Sabyasachi Mukherjee}
\address{School of Mathematics, Tata Institute of Fundamental Research, 1 Homi Bhabha Road, Mumbai 400005, India}
\email{sabya@math.tifr.res.in,mukherjee.sabya86@gmail.com}
\thanks{Both authors were  supported by  the Department of Atomic Energy, Government of India, under project no.12-R\&D-TFR-5.01-0500 as also  by an endowment of the Infosys Foundation. MM was also supported in part by a DST JC Bose Fellowship. SM was supported in part by SERB research project grant MTR/2022/000248.} 

\date{\today}

\begin{document}

\title[Simultaneous Uniformization and Algebraic Correspondences]{Simultaneous Uniformization and\\ Algebraic Correspondences}

\begin{abstract}
We prove a generalization of  Bers' simultaneous uniformization theorem in the world of algebraic correspondences. More precisely, we construct algebraic correspondences that simultaneously uniformize a pair of non-homeomorphic genus zero orbifolds. We also present a complex-analytic realization of the Teichm{\"u}ller space of a punctured sphere in the space of correspondences.
\end{abstract}

\maketitle

\setcounter{tocdepth}{1}
\tableofcontents

\section{Introduction}\label{intro_sec}

Algebraic correspondences on the Riemann sphere $\widehat{\C}$, viewed as dynamical systems, include rational dynamics and actions of Kleinian groups. For the purposes of this paper, an \emph{algebraic correspondence} on $\widehat{\C}$ is a finite-to-finite multi-valued map with holomorphic local branches. Equivalently, let $\pi_1$ and $\pi_2$ be the two canonical projections from $\widehat{\mathbb{C}} \times \widehat{\mathbb{C}}$ to the first and second factors respectively, and $Z$ be a  codimension $1$  subvariety in $\widehat{\mathbb{C}} \times \widehat{\mathbb{C}}$ containing no fibers of $\pi_1$ and $\pi_2$; then $\pi_2 \circ \pi_1^{-1}$ is said to be an algebraic correspondence from the first to the second factor. Algebraic correspondences, viewed as multi-valued maps, can be iterated in the obvious way (cf. \cite{BP01}).

%We shall only need the following special class of correspondences.
%Let $\pi_1$ and $\pi_2$ be the two canonical projections from $\widehat{\mathbb{C}} \times \widehat{\mathbb{C}}$
% to the first and second factors respectively. The underlying variety  of the correspondence we want to construct on $\widehat{\mathbb{C}}$ is a  codimension 1  subvariety $Z$ in $\widehat{\mathbb{C}} \times \widehat{\mathbb{C}}$ containing no fibers of $\pi_1$ and $\pi_2$. Then $\pi_2 \circ \pi_1^{-1}$ is said to be an \emph{algebraic correspondence} from the first to the second factor.}

That rational dynamics and actions of Kleinian groups can be unified in the framework of correspondences was observed by Fatou in the 1920s \cite{Fatou29}. Bullett and Penrose constructed the first examples of correspondences that combine the dynamics of quadratic rational maps and the modular group \cite{BP94,BL20}.

\subsection*{Virtual orbit equivalence, and a combination framework}
 An \emph{orbit equivalence} framework for combining/mating Fuchsian groups (or surfaces) with polynomials was developed in \cite{MM1}. The key idea of \cite{MM1} was to replace a Fuchsian group with a piecewise M{\"o}bius circle map (called a \emph{mateable map}) that retains some of the key features of the group (namely, it is orbit equivalent to the group) and is compatible with polynomial dynamics (that is to say, it is topologically conjugate to $z^d$, for some $d\geq 2$). Using the notion of mateable maps, combination theorems for these objects were proved.   Principal examples of mateable maps associated with Fuchsian groups are given by \emph{Bowen-Series} (Section~\ref{sec-bs}) and \emph{higher Bowen-Series} (Section~\ref{sec-hbs}) maps of Fuchsian punctured sphere groups.

The above framework was extended to \emph{virtual orbit equivalences} in \cite{MM2}. More precisely, one looks at a finite index subgroup $\Gamma_0'$  of the original Fuchsian group $\Gamma_0$, and shows that the Bowen-Series map of $\Gamma_0'$ (which is orbit equivalent to $\Gamma_0'$) acting on the circle admits a factor dynamical system, called the \emph{factor Bowen-Series map}. The restriction of a factor Bowen-Series map to the circle is topologically conjugate to $z^d$, for some $d\geq 2$, and hence it can be conformally mated with complex polynomials. 
Algebraic descriptions of these conformal matings were also given (initially under a real-symmetry assumption, which was later dropped in \cite{LLM24}). This was used to construct correspondences on (possibly nodal) Riemann spheres. These correspondences are matings of complex polynomials and genus zero orbifolds.

The above construction was used in \cite{MM2} to construct holomorphic embeddings of \emph{Bers slices} of genus zero orbifolds in the space of algebraic correspondences such that the correspondences are matings of the corresponding surfaces and the polynomial $z^d$.

\subsection*{Statement of the main theorem}
In this article, we show that this mating framework also has purely Teichm{\"u}ller-theoretic consequences. Specifically, we establish the following generalization of the \emph{Bers' Simultaneous Uniformization Theorem} (cf. \cite{Ber60}).

\begin{theorem}\label{simult_unif_corr_thm}
Let $\Sigma_1, \Sigma_2$ be (possibly non-homeomorphic) hyperbolic orbifolds of genus zero with arbitrarily many (at least one) punctures, at most one order two orbifold point, and at most one order $\nu\geq 3$ orbifold point with $d(\Sigma_1)=d(\Sigma_2)$.
Then, there exists an algebraic correspondence $\mathfrak{C}$ (on a possibly noded Riemann surface) which acts via conformal automorphisms on its regular set $\Omega(\mathfrak{C})$, and the quotient $\Omega(\mathfrak{C})/\mathfrak{C}$ is biholomorphic to the disjoint union of $\Sigma_1$ and $\Sigma_2$.
\end{theorem}
\noindent (See Subsection~\ref{fbs_sec} for the definition of $d(\Sigma)$.)
\smallskip

A word about the general framework in which Theorem~\ref{simult_unif_corr_thm} belongs is in order. It is clear that a Kleinian group with a (symmetric) generating set consisting of $g_1, \cdots, g_k$ is equivalent to the data of a reducible polynomial $P(x,y)=\Pi_i (y-g_i(x))$. The codimension $1$ variety $\{P(x,y)=0\}$ gives rise to an algebraic correspondence, whose branches are given by $g_1, \cdots, g_k$. 

However, the framework of correspondences is  manifestly more flexible; it includes general rational functions and dynamical systems exhibiting features of rational maps and Fuchsian groups simultaneously (cf. \cite{MM2,LLM24}). In particular, uniqueness of a correspondence $\mathfrak{C}$ combining two orbifolds $\Sigma_1, \Sigma_2$ in the sense of Theorem~\ref{simult_unif_corr_thm} is too much to hope for (see the higher Bowen-Series maps in \cite{MM1} for instance).

In the special case  where $\Sigma_1, \Sigma_2$ are \emph{homeomorphic} hyperbolic orbifolds satisfying the hypotheses of Theorem~\ref{simult_unif_corr_thm}, the theorem reduces to  Bers' Simultaneous Uniformization Theorem (by producing a generating set for the quasi-Fuchsian group that simultaneously uniformizes $\Sigma_1, \Sigma_2$). It is in this sense that  Theorem~\ref{simult_unif_corr_thm} is a generalization.

\subsection*{A sketch of the proof}
Let us now briefly mention the main ideas that go into the proof of Theorem~\ref{simult_unif_corr_thm}. The key source of flexibility in the proof is the virtual orbit equivalence mating framework mentioned above. Specifically, this mating framework does not require the surfaces $\Sigma_1$ and $\Sigma_2$, appearing in Theorem~\ref{simult_unif_corr_thm}, to be homeomorphic; it is enough to assume that the factor Bowen-Series maps associated with $\Sigma_1$ and $\Sigma_2$ have equal degree. With this condition in place, we proceed to construct the conformal mating of the associated factor Bowen-Series maps (see Proposition~\ref{conf_mating_exists_prop}). 

We then inspect the conformal mating $F$ described in the previous paragraph on the boundary of its natural domain of definition, and observe that it is a \emph{boundary involution}; i.e., an orientation-reversing involution on a piecewise real-analytic curve (see Proposition~\ref{conf_mating_b_inv_prop}).
Using an algebraic characterization of such meromorphic maps derived in \cite{MM2,LLM24}, we write the conformal mating as an explicit algebraic function. Concretely, the conformal mating $F$ can be modeled by the standard involution $\eta(z)=1/z$ via a rational map $R$ of the Riemann sphere.  The rest of the proof of Theorem~\ref{simult_unif_corr_thm} follows the strategy of \cite{MM2}. The algebraic description of $F$ allows us to promote it to an algebraic correspondence, which is defined as compositions of $\eta$ and the local deck transformations of $R$. Finally, arguments similar to the ones used in \cite{MM2} show that the correspondence uniformizes the surfaces $\Sigma_1, \Sigma_2$ simultaneously, in a suitable sense (see Section~\ref{sec-proof_mainthm} for details).

\subsection*{Organization of the paper}
The paper is organized as follows.
Section~\ref{bs_hbs_fbs_sec} surveys the orbit equivalence and virtual orbit equivalence mating framework between genus zero orbifolds and polynomial dynamics leading up to the construction of algebraic correspondences arising as combinations of the corresponding Fuchsian groups and polynomials. In Section~\ref{qf_corr_sec}, we prove one of the main new results of this article. In particular, we show that the virtual orbit equivalence mating framework is `less demanding' in that it allows one to manufacture a semi-global complex-analytic map of the Riemann sphere that combines a pair of topologically nonequivalent genus zero orbifolds. Here we also characterize this \emph{conformal mating} as an explicit algebraic function, and then globalize it to obtain the desired algebraic correspondences. In the final Section~\ref{bers_sec}, we use Theorem~\ref{simult_unif_corr_thm} to construct a holomorphic embedding of the Teichm{\"u}ller space of a puncture sphere into the space of algebraic correspondences, each of which is generated by a M{\"o}bius involution and the local deck transformations of a rational map (see Theorem~\ref{bers_embedding_thm}).

%In \cite{MM1}, we introduced the orbit equivalence mating framework: this says that the grand orbit of $z \to z^d$ equals the orbit of $z \to A(z)$, where $A$ represents the Bowen Series map for a sphere with 3 or more punctures. In \cite{MM2} this framework was relaxed slightly to a virtual orbit equivalence mating framework where we are allowed to pass to a factor dynamical system and, in particular, a factor Bowen-Series map. As mentioned above, Section~\ref{bs_hbs_fbs_sec} provides details.

%With this background in place, we find a simple topological/combinatorial condition that allows two orbifolds of different topological types to be virtually orbit equivalent. We then proceed to mate the associated factor Bowen Series maps (see Proposition~\ref{conf_mating_exists_prop}).}

%Our strategy in \cite{MM1} was to find a mating between a polynomial and a Bowen-Series map. In \cite{MM2}, we investigated the mating on the boundary of its natural domain of definition and observed that it was an involution. As in \cite{MM2}, we inspect the mating described in the previous paragraph on the boundary of its natural domain of definition. We observe that it is an involution (see Proposition~\ref{conf_mating_b_inv_prop}). The rest of the proof of Theorem~\ref{simult_unif_corr_thm} follows the strategy of \cite{MM2}. The involution on the boundary of definition of the mating allows us to use the standard involution $\eta(z) = 1/z$ to construct a correspondence from the mating (see Section~\ref{sec-proof_mainthm} for details.

\section{Mateable and virtually mateable maps}\label{bs_hbs_fbs_sec} This section surveys \cite{MM1} (particularly Sections 2--4 of that paper), where we introduced mateable maps and gave the first set of examples of such maps.
See \cite{MMsurvey} for a more detailed survey of this topic.

\subsection{Mateable maps}\label{sec-mateable}
\begin{defn}\label{defn-mateable}
Let $A:\mathbb{S}^1\to\mathbb{S}^1$ be a  continuous piecewise analytic map. Then $A$  is said to be a \emph{mateable} map corresponding to a Fuchsian group $\Gamma$ if the following hold:
	\begin{enumerate}[\upshape ({{M}}-1)]
		\item\label{mateable-oe} $A$ is orbit equivalent to $\Gamma$; i.e., the grand orbits of $A$ equal the orbits of $\Gamma$.
		\item\label{mateable-exp} $A$ is an expansive covering map of degree $d$ greater than one.
		\item\label{mateable-markov} $A$ is Markov; i.e., the maximal connected subsets of $\mathbb{S}^1$ on which $A$ is genuinely analytic give a Markov partition of $\mathbb{S}^1$ for $A$.
		\item\label{mateable-asym_hyp} No periodic break-point of $A$ is asymmetrically hyperbolic; i.e., at such break-points, the multipliers on the two sides need to be equal.
	\end{enumerate}
\end{defn}  

Here,
\begin{enumerate}
\item Condition~\ref{mateable-exp} is equivalent to saying that 
$A$ is topologically conjugate to the standard degree $d$ map
$z\mapsto z^d$ on $\mathbb{S}^1$.
\item Condition~\ref{mateable-oe} furnishes a rather soft \emph{dynamical} compatibility between the Fuchsian group $\Gamma$ and 
$z\mapsto z^d$. Indeed, since $z\mapsto z^d$ and $A$ are topologically  orbit-equivalent by the above observation, it follows from Condition~\ref{mateable-oe} that 
$z\mapsto z^d$ and $\Gamma$ have the same (grand) orbits after a topological change of coordinates.
\item Condition~\ref{mateable-markov} furnishes a \emph{combinatorial}  compatibility between the Fuchsian group $\Gamma$ and 
$z\mapsto z^d$ by demanding that the pieces of $A$ give a Markov partition for $z\mapsto z^d$ after a topological change of coordinates.
\item Condition~\ref{mateable-asym_hyp} ensures that locally
(at break-points) the multipliers on the left and right are consistent with the behavior of $z\mapsto z^d$.
\end{enumerate}

Thus, the conditions of Definition~\ref{defn-mateable}
impose minimalistic conditions for conformal mateability of $A$ and 
$z\mapsto z^d$. Surprisingly, it turns out that the conditions of Definition~\ref{defn-mateable} are sufficient as shown in \cite[Proposition~2.18]{MM1} (see below).
\medskip

\noindent {\bf Canonical extension and fundamental domain of a piecewise M\"obius map.}\\
	Let $A$ be a continuous 
	piecewise M\"obius map on the circle. Let $\D$ denote the unit disk.
	Let $J_1, \cdots, J_k$
	be the \emph{pieces} of $A$; i.e.,  $J_1, \cdots, J_k$ are a circularly ordered sequence of closed intervals with disjoint interiors such that
	\begin{enumerate}
 \item $\displaystyle\bigcup_{j=1}^k J_i = \mathbb{S}^1$,
		\item $J_j \cap J_{j+1} = \{x_{j+1}\}$ (we assume here that the indices are taken mod~$k$).
		\item $A |_{J_j} = g_j$.
	\end{enumerate}
	Let $\gamma_j$ be the bi-infinite hyperbolic geodesic in $\D$ (equipped with the standard hyperbolic metric) between $x_{j}, x_{j+1}$. Let $\mathcal{D}_j \subset \overline{\D}$ denote the closed region bounded by $J_j$ and $\gamma_j$. 
 \begin{defn}\label{defn-canextn}
     The \emph{canonical extension} of $A$, denoted by
	$\widehat{A}$,  
	is defined on $\cD:=\displaystyle\bigcup_{j=1}^k \mathcal{D}_j$ by $\widehat A = g_j$ on 
 $\mathcal{D}_j$.

 The set $\mathcal{D}$ is called the 
 \emph{canonical domain of definition} of $\widehat{A}$ in $\overline{\D}$.

 The open ideal polygon bounded by the bi-infinite hyperbolic geodesics $\gamma_j$ is called the
 \emph{fundamental domain} of the  piecewise M\"obius map $A$ and is denoted by $R$. \\
 \end{defn}

\noindent {\bf Polynomial dynamics.}\\
Now, let $P$ be 
 a complex polynomial of degree $d>1$ (for our purposes, the qualitative 
 features of $P$ will be similar to those of  $z \mapsto z^d$).

 \begin{defn}
      The \emph{filled Julia set} $\mathcal{K}(P)$ is defined to be the completely invariant set of all points whose forward orbits under $P$ are bounded. The polynomial $P$ is said to be \emph{hyperbolic} if all of its critical points converge to attracting cycles under forward iteration. 
 \end{defn}

 It is a classical fact of complex dynamics that the
 set of all hyperbolic polynomials of degree $d$ ($d>1$) is open in the parameter space. A connected component of such hyperbolic polynomials in the parameter space is called a \emph{hyperbolic component}. 
 
 Let $\cH_d$ denote the  hyperbolic component  containing the map $z \mapsto z^d$. We refer to  $\cH_d$ as the \emph{principal hyperbolic component}. 
 For any $f \in \cH_d$, the filled Julia set  is a (closed) quasidisk.
 Further,  the dynamics of $f$ on its Julia set is quasi-symmetrically conjugate to the action of $z \mapsto z^d$ on $\mathbb S^1$. Thus, the qualitative 
 features of $P$ are similar to those of  $z \mapsto z^d$.
 We are now ready to state the proposition that asserts that the conditions of Definition~\ref{defn-mateable} suffice. We refer the reader to \cite[\S 2.3]{MM1} for the definition of conformal mating.

\begin{prop}[Mateable maps are mateable]\cite[Proposition~2.18]{MM1}\label{pro-mateable}
	Let $A:\mathbb{S}^1\to\mathbb{S}^1$ be a mateable map of degree $d$ in the sense of Definition~\ref{defn-mateable}.
 Let $P\in\mathcal{H}_d$.
	Then, $\widehat{A}:\mathcal{D}\to\overline{\D}$ and $P:\mathcal{K}(P)\to\mathcal{K}(P)$ are conformally mateable. 
\end{prop}

\subsection{Bowen-Series Maps}\label{sec-bs}
It remains to furnish examples of mateable maps in the sense of Definition~\ref{defn-mateable}. The first examples of mateable maps come from Bowen-Series maps of Fuchsian groups corresponding to punctured spheres. We briefly recall this, and refer the reader to \cite[Section 3.2]{MMsurvey} for further details.

Let $\Sigma_d$ denote the $(d+1)-$punctured sphere. We construct a specific $2d-$sided ideal polygon in the unit disk symmetric about the $x-$axis. Let $1=z_0, \cdots, z_d=-1$ denote the $2d$-th roots of unity on the upper semi-circle arranged counter-clockwise. The vertices of the ideal $2d$-gon are given by $z_0,\cdots, z_d, \overline{z_1}, \cdots, \overline{z_{d-1}}$. The side-pairing transformations take the edge (bi-infinite geodesic) between $\overline{z_i}, \overline{z_{i+1}}$ to the edge (bi-infinite geodesic) between $z_i, z_{i+1}$ for all the middle edges; i.e., $i=1,\cdots,d-2$. The edge between $z_0, \overline{z_1}$ is taken to the edge between 
$z_0, z_{1}$. Similarly, the edge between  $\overline{z_{d-1}}$ and $z_{d}$
is taken to the edge between $z_{d-1}$ and $z_{d}$. See \cite[Figure 2]{MMsurvey} for a diagram illustrating this situation. Let $\sigma_1,\cdots, \sigma_{d}$ denote the associated M{\"o}bius transformations on the unit disk $\D$.
The associated Bowen-Series map $A_{\mathrm{BS},d}:\mathbb{S}^1 \to \mathbb{S}^1$ is the piecewise analytic map defined by
\begin{enumerate}
    \item $\sigma_i^{-1}$ on the arc joining $z_{i-1}$ and $z_{i}$ for $i=1,\cdots, d$,
    \item $\sigma_i$ on the arc joining $\overline{z_{i-1}}$ and $\overline{z_{i}}$ for $i=1,\cdots, d$.
\end{enumerate}
Let $\Gamma_0$ be the group generated by $\sigma_1,\cdots, \sigma_{d}$. For any marked group $\Gamma\in\mathrm{Teich}(\Sigma_d)$ (where $\mathrm{Teich}(\Sigma_d)$ stands for the Teichm\"uller space of $\Sigma_d$), the Bowen-Series map $A_{\mathrm{BS},\Gamma}:\mathbb{S}^1\to\mathbb{S}^1$ is defined as the conjugate of $A_{\mathrm{BS},d}:\mathbb{S}^1 \to \mathbb{S}^1$ by the quasiconformal homeomorphism that conjugates the marked group $\Gamma_0$ to the marked group $\Gamma$.

The following summarizes the properties of the above Bowen-Series maps.

\begin{theorem}\cite[Proposition~3.3, Theorem~3.7]{MM1}\label{thm-BSmateable}
Suppose that $d\geq 2$, so that $\Sigma_d$ has at least 3 punctures. Then for any marked group $\Gamma\in\mathrm{Teich}(\Sigma_d)$, the Bowen-Series map
$A_{\mathrm{BS},\Gamma}:\mathbb{S}^1 \to \mathbb{S}^1$ is a degree $2d-1$ mateable map in the sense of Definition~\ref{defn-mateable}. In particular,  $A_{\mathrm{BS},\Gamma}:\mathbb{S}^1 \to \mathbb{S}^1$ is orbit equivalent to the Fuchsian group $\Gamma$.

 Let $P\in\mathcal{H}_{2d-1}$.
 Then, the canonical extension (cf.\ Definition~\ref{defn-canextn}) $\widehat{A}_{\mathrm{BS},\Gamma}:\mathcal{D}_{A_{\mathrm{BS},\Gamma}}\to\overline{\disk}$ and $P:\mathcal{K}(P)\to\mathcal{K}(P)$ are conformally mateable.
\end{theorem}

\subsection{Higher Bowen-Series Maps}\label{sec-hbs}
We will now describe another class of mateable maps associated with Fuchsian punctured sphere groups.
\smallskip

\noindent {\bf Higher degree map without folding.}\\
Recall the notions of canonical extension and fundamental domain $R$ from Definition~\ref{defn-canextn}.
A \emph{diagonal} of $R$ is a bi-infinite geodesic in $R$ joining a pair of non-adjacent vertices in $R$. Let $A: \mathbb S^1 \to \mathbb S^1$ be
a	piecewise M\"obius  map with fundamental domain $R$. We say that $A$ has a \emph{diagonal fold} if there
	exist consecutive edges $\alpha_1, \alpha_2$ of $\partial R$
	and a diagonal $\delta$ of $R$ such that the canonical extension $\widehat{A}$ maps $\alpha_1, \alpha_2$  to
 $ \delta$. 
 
 Let $a_1, a_2$ and $ a_2, a_3$ be the endpoints of $\alpha_1$, $\alpha_2$ respectively. Also, let $p, q$ be the endpoints of $\delta$. Since  $\widehat{A}$ is continuous on $\cD$, it follows that
	$A(a_1)=p=A(a_3)$ and $A(a_2)=q$.

 \begin{defn}\label{def-hdm}
   Let $A$ be a  piecewise M\"obius map from $\mathbb S^1$ to itself. 
   $A$ is said to be a \emph{higher degree map without folding} if it satisfies the following.
	\begin{enumerate}
		\item There exists an  ideal polygon $R_0\subset R$ such that  the (cyclically ordered) edges 
		$\delta_1, \cdots, \delta_l$ of $R_0$ are   diagonals of $R$.  
		\item If $p$ is an ideal vertex of $R_0$, then it is fixed under $A$; i.e., $A(p)=p$.
		\item Every edge $\alpha$ of $R$ is mapped by $A$ to  one of the sides
		$\delta_1, \cdots, \delta_l$ of~$R_0$.
		\item $A$ has no diagonal folds. 
	\end{enumerate} 
	The ideal polygon $R_0$ is called the \emph{inner domain} of $A$.
 \end{defn}

We assume as usual the sides
$\alpha_1, \cdots, \alpha_k$ of $R$ are cyclically ordered. Then we observe that consecutive edges
$\alpha_i, \alpha_{i+1}$ are mapped to consecutive edges of the inner domain $R_0$. The ordering, may however, be reversed, i.e.\ clockwise cyclic ordering may go to
counterclockwise cyclic ordering and vice versa.
Thus, we obtain
a continuous map $\widehat{A}: \partial R \to \partial R_0$.
After adding on the ideal endpoints of $R$ and $R_0$, we note that $\widehat{A}: \partial R \to \partial R_0$ has a well-defined
degree $d$. Thus, any edge of $R_0$ has exactly $|d|$ pre-images (this is the place where we use the `no folds' hypothesis). 

\begin{defn}\label{def-degree}
    We refer to $|d| >1$ as the \emph{polygonal degree} of $A$.
\end{defn}

\noindent {\bf Higher Bowen-Series Maps.}\\
Next, we fix a regular ideal 2d-gon $W$ as in Section~\ref{sec-bs}. This will be the fundamental domain for a base Fuchsian group $\Gamma_0$ isomorphic to $\pi_1(\Sigma_d)$, where (recall) $\Sigma_d$ denotes a $(d+1)-$punctured sphere.
Concretely, assume that the ideal vertices of $W$ are 
 the $2d$-th roots of unity. Let the vertices of $W$ on the
 lower semi-circle be numbered $1=1_-$, $2_-$ $\cdots$, $(d+1)_-=d+1$ in counterclockwise order. Let the vertices of $W$ on the
upper semi-circle be numbered $1, 2, \cdots, d+1$ in clockwise order (see Figure~\ref{fig-hbs}, see also \cite[Figure 3]{MM1}).

As in Section~\ref{sec-bs}. the generators of $\Gamma_0$ are given by $\sigma_1, \cdots, \sigma_{d} $, where $\sigma_i$ takes the edge $\overline{i_- (i+1)_-}$ to the  edge $\overline{i (i+1)}$ (here on, for $a,b\in\mathbb{S}^1$, the bi-infinite hyperbolic geodesic in $\D$ with ideal endpoints at $a,b$ will be denoted by $\overline{ab}$).

\begin{figure}[ht!]
\captionsetup{width=0.98\linewidth}
\begin{tikzpicture}
\node[anchor=south west,inner sep=0] at (0,0) {\includegraphics[width=0.5\linewidth]{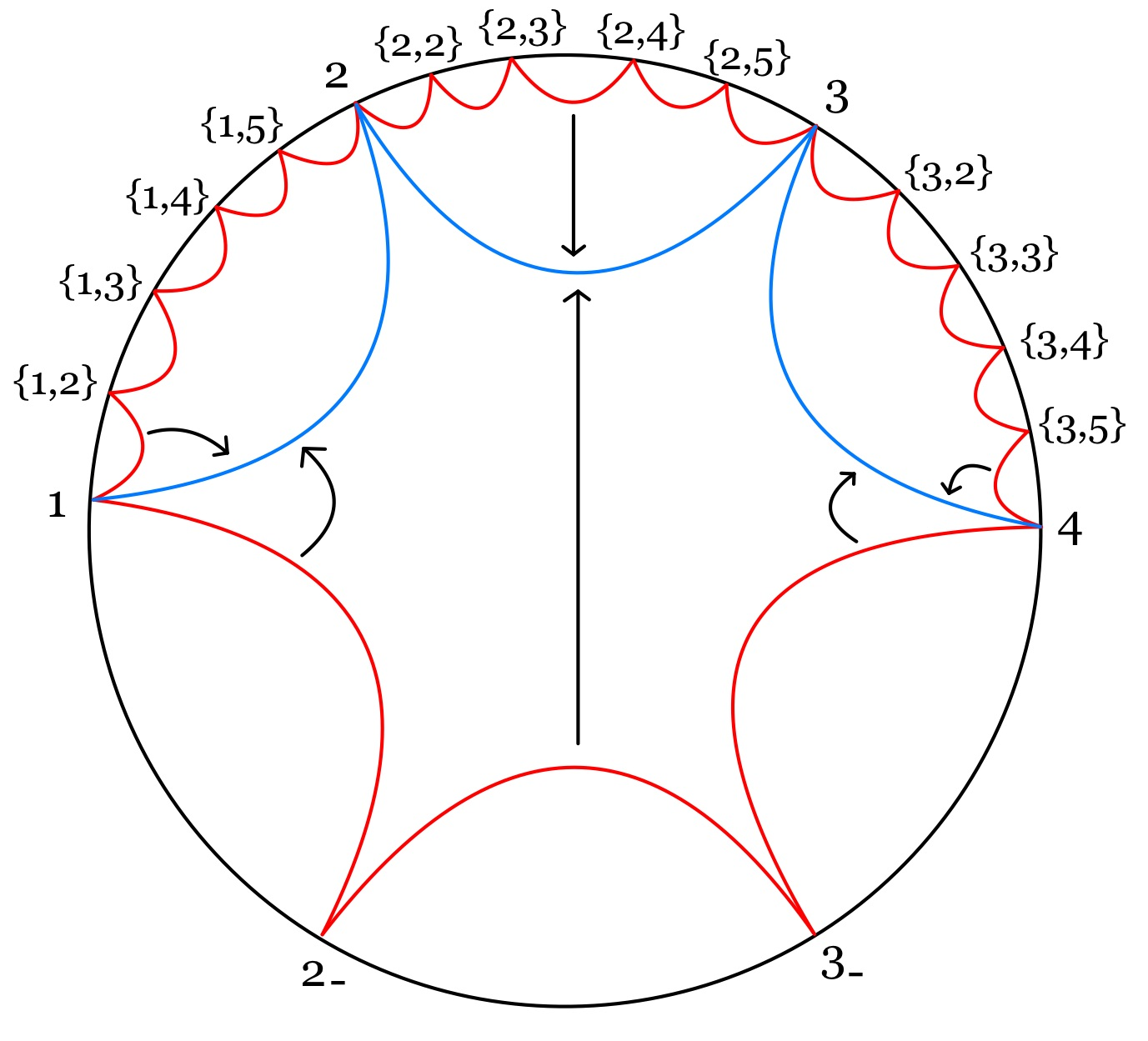}};
\node[anchor=south west,inner sep=0] at (7,0) {\includegraphics[width=0.44\linewidth]{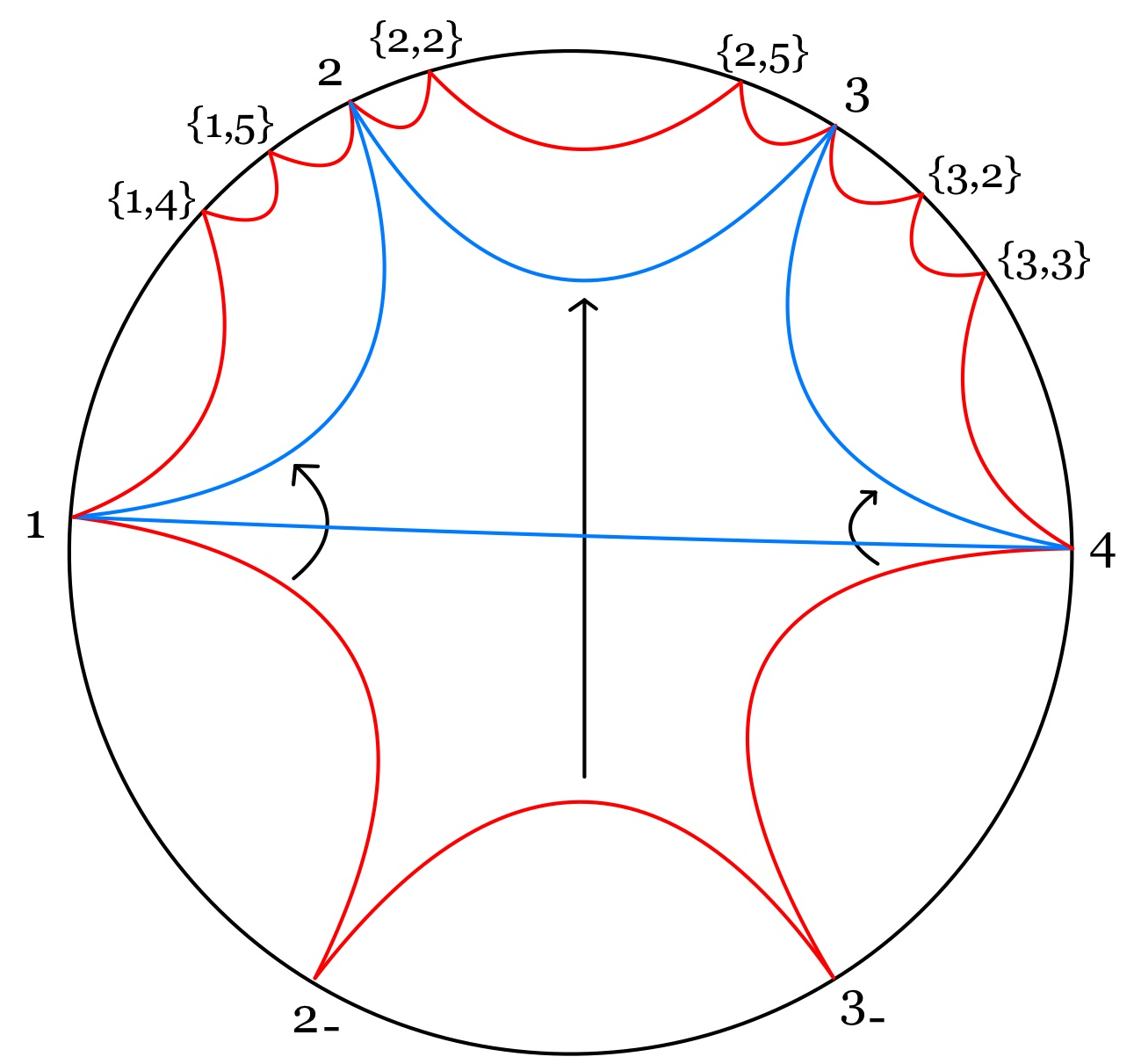}};

\node at (4.4,2.9) {\begin{small}$\sigma_3$\end{small}};
\node at (3.5,3) {\begin{small}$\sigma_2$\end{small}};
\node at (2.1,2.9) {\begin{small}$\sigma_1$\end{small}};
\node at (1.25,3.6) {\begin{small}$\sigma_1^{-1}$\end{small}};
\node at (3.5,4.8) {\begin{small}$\sigma_2^{-1}$\end{small}};
\node at (5.12,3.38) {\begin{small}$\sigma_3^{-1}$\end{small}};
\node at (10.94,2.75) {\begin{small}$\sigma_3$\end{small}};
\node at (10.1,3) {\begin{small}$\sigma_2$\end{small}};
\node at (8.84,2.78) {\begin{small}$\sigma_1$\end{small}};
\end{tikzpicture}
\caption{Fundamental domains for $A_{\Gamma_0, \mathrm{aux}}$ and $A_{\Gamma_0, \mathrm{hBS}}$: $4$ punctures.}
\label{fig-hbs}
\end{figure}

Next, we need to interpolate vertices between vertices $i, i+1$ on the
upper semi-circle. Between vertices $i, i+1$, we interpolate $2d$ vertices 
given by the vertices of $\sigma_i.W$. Note that $\sigma_i.W \cap W=\overline{i (i+1)}$. The resulting new vertices are labeled $\{i,2\}, \{i,3\}, \cdots,
	\{i,2d-1\}$ in clockwise order. 
 
We will now define an auxiliary piecewise M{\"o}bius map $A_{\Gamma_0, \mathrm{aux}}$ having 
$$
R=\Int{\left(W \cup \bigcup_{i=1}^d \sigma_i.W\right)}
$$ 
as its fundamental domain.
Note that $\overline{i (i+1)}$, $i=1,\cdots,d$, are diagonals of $R$.

For two ideal boundary points $a,b$ of $R$, let $\arc{ab}$ denote the maximal arc of $\mathbb S^1$ with endpoints at $a,b$ and no internal break-points; i.e., $\arc{ab}$ has endpoints at the break-points $a, b$, and there do not exist any other ideal boundary points of $R$ in the arc. On  $\arc{\ i_- (i+1)_-\ }$, define $$A_{\Gamma_0, \mathrm{aux}}=\sigma_i, \, i=1, \cdots, d.$$
 Note that $A_{\Gamma_0, \mathrm{aux}}$ maps 
 $\arc{\ i_- (i+1)_-\ }$ to the closure of the complement of  the arc $\arc{\ i (i+1)\ }$.

Next, for  $i=1, \cdots, d$, and on each of the $d$ short arcs $\arc{\ \{i,j\} \{i,j+1\}\ }$ for $i \leq j \leq i+d-1$ between $i, i+1$, define 
$$A_{\Gamma_0, \mathrm{aux}}=\sigma_i^{-1}.$$
Note that $A_{\Gamma_0, \mathrm{aux}}$ maps $\displaystyle\left(\bigcup_{j=i}^{i+d-1}\arc{\ \{i,j\} \{i,j+1\} }\right)$
to the entire top semicircle between $1$ and $d+1$. We are implicitly identifying $\{i,1\}$ with $i$ and $\{i, i+2d\}$ with $i+1$ here. Further, for $i \leq j \leq i+d-1$, $A_{\Gamma_0, \mathrm{aux}}$ maps the clockwise arc from $\{i,j\}$ to $\{i,j+1\}$ onto the clockwise arc from $j$ to $j+1$.

    For $i\in\{2,\cdots, d\}$ and $1\leq j\leq i-1$, let $j=i-s$, so that
	$1\leq s \leq i-1$. Define $$A_{\Gamma_0, \mathrm{aux}}=\sigma_s \circ \sigma_i^{-1}$$ on $\arc{\ \{i,j\} \{i,j+1\}\ }.$ Note that $A_{\Gamma_0, \mathrm{aux}}$ maps $\arc{\{i,j\} \{i,j+1\}}$ to the long arc from $s$ to $s+1$ in a counterclockwise sense.

For $i\in\{1,\cdots, d-1\}$ and $i+d \leq j\leq 2d-1$, let $j=i+d+t$. Thus, 
	$0\leq t \leq d-1-i$. Define $$A_{\Gamma_0, \mathrm{aux}}=\sigma_{d-t} \circ \sigma_i^{-1}$$
 on $\arc{\ \{i,j\} \{i,j+1\}\ }$. Hence, $A_{\Gamma_0, \mathrm{aux}}$ maps $\arc{ \{i,j\} \{i,j+1\}}$ to the long arc from $d-t$ to $d+1-t$
	in a counterclockwise sense.
 
We observe that
$A_{\Gamma_0, \mathrm{aux}}$ fixes the vertex $i$ for all $i=1,\cdots, d+1$.

 Define $A_{\Gamma_0, \mathrm{hBS}}$ to be the minimal piecewise M\"obius map equaling $A_{\Gamma_0, \mathrm{aux}}$ on $\mathbb S^1$. Thus,  $A_{\Gamma_0, \mathrm{hBS}}$ equals $A_{\Gamma_0, \mathrm{aux}}$ pointwise; however, all superfluous break-points have been removed in passing from $A_{\Gamma_0, \mathrm{aux}}$ 
 to $A_{\Gamma_0, \mathrm{hBS}}$.
 
 Let $\widehat{A}_{\Gamma_0, \mathrm{hBS}}$ be the canonical extension of $A_{\Gamma_0, \mathrm{hBS}}$. It is easy to check that $\widehat{A}_{\Gamma_0, \mathrm{hBS}}$ is a higher degree map without folding in the sense of Definition~\ref{def-hdm}. The inner polygon for this higher degree map without folding is the ideal polygon with vertices at $1, 2,\cdots, d+1$.

\begin{defn}\label{defn-hbs}
	We call the piecewise M{\"o}bius Markov map $A_{\Gamma_0, \mathrm{hBS}}$ the \emph{higher Bowen-Series map} of $\Gamma_0$ (associated with the fundamental domain $W$). For a marked group $\Gamma\in\mathrm{Teich}(\Sigma_d)$, the \emph{higher Bowen-Series map} $A_{\Gamma,\mathrm{hBS}}:\mathbb{S}^1\to\mathbb{S}^1$ is defined as the conjugate of $A_{\Gamma_0,\mathrm{hBS}}:\mathbb{S}^1 \to \mathbb{S}^1$ by the quasiconformal homeomorphism that conjugates the marked group $\Gamma_0$ to the marked group $\Gamma$.
\end{defn}

One of the main theorems of \cite{MM1} can now be summarized as follows:

\begin{theorem}\label{theorem-bshbsmateable}
   Let $A$ be a higher Bowen-Series map (in the sense of Definition~\ref{defn-hbs}) of a Fuchsian group uniformizing a punctured sphere. Then the canonical extension $\widehat{A}$ of $A$ 
   can be conformally mated with polynomials lying in the principal hyperbolic component of degree $d^2=\deg(A\vert_{\mathbb{S}^1})$.
\end{theorem}

\subsection{Virtually mateable maps}
We would now like to generalize Definition~\ref{defn-mateable}  to allow mild discontinuities.
\begin{defn}\label{defn-vmateable}
Let $A:\mathbb{S}^1\to\mathbb{S}^1$ be a  continuous piecewise analytic map.  Then $A$  is said to be a \emph{virtually mateable} map corresponding to a Fuchsian group $\Gamma$ if the following hold:
	\begin{enumerate}[\upshape ({{VM}}-1)]
		\item\label{vmateable-oe} $A$ is a factor of a possibly discontinuous circle endomorphism $\widetilde{A}$ such that the latter is orbit equivalent to a finite index subgroup of $\Gamma$.
		\item\label{vmateable-exp} $A$ is an expansive covering map of degree $d$ greater than one.
		\item\label{vmateable-markov} $A$ is virtually Markov; i.e., there exists $n \in \natls$ such that the  $n-$fold preimages of maximal connected subsets of $\mathbb{S}^1$ on which $A$ is genuinely analytic give a Markov partition of $\mathbb{S}^1$ for $A$.
		\item\label{vmateable-asym_hyp} No periodic break-point of $A$ is asymmetrically hyperbolic; i.e., at such break-points, the multipliers on the two sides need to be equal.
\end{enumerate}
\end{defn}
\begin{prop}[Virtually mateable maps are mateable]\label{prop-vmateable}
	Let $A:\mathbb{S}^1\to\mathbb{S}^1$ be a virtually mateable map of degree $d$ in the sense of Definition~\ref{defn-vmateable}.
 Let $P\in\mathcal{H}_d$.
	Then, $\widehat{A}:\mathcal{D}\to\overline{\D}$ and $P:\mathcal{K}(P)\to\mathcal{K}(P)$ are conformally mateable. 
\end{prop}

\begin{proof}
    The proof is exactly the same as that of 
    Proposition~\ref{pro-mateable}. This is because the key analytical tool used in its proof, the David extension theorem, does not need orbit equivalence.
\end{proof}

\subsection{Factor Bowen-Series maps}\label{fbs_sec} In this section, we provide examples of virtually mateable maps that are not mateable maps. Nevertheless, Proposition~\ref{prop-vmateable} applies to these examples.

The following class of orbifolds and the associated circle endomorphisms were introduced in \cite{MM2}.\\
 \smallskip
 
 $\mathfrak{F}:=$ hyperbolic orbifolds $\Sigma$ of genus zero with
	\begin{enumerate}
	\item at least one puncture,
	\item at most one order two orbifold point, and
	\item at most one order $\nu\geq 3$ orbifold point.
	\end{enumerate}
We set
\begin{align*}
n=
\begin{cases}
\nu \quad \textrm{if } \Sigma\in\mathfrak{F}\ \textrm{ has an order } \nu\geq 3 \textrm{ orbifold point},\\
1 \quad \mathrm{otherwise}.
\end{cases}
\end{align*}
An orbifold $\Sigma\in\mathfrak{F}$ admits an $n-$fold cyclic cover $\widetilde{\Sigma}$, which is obtained by skewering the surface $\Sigma$ along an infinite geodesic connecting the order $\nu$ orbifold point and a cusp, and gluing $n$ copies of it cyclically. If $\Sigma$ does not have an order $\nu\geq 3$ orbifold point, then $\widetilde{\Sigma}=\Sigma$.
It is easily seen from the above construction that if $\Sigma$ has $\delta_1\geq1$ punctures and $\delta_2\in\{0,1\}$ order two orbifold points, then $\widetilde{\Sigma}$ is a genus zero orbifold with $n(\delta_1-1)+1$ punctures and $n\delta_2$ order two orbifold points. 

When $n\geq 3$, the Fuchsian group $\Gamma$ that uniformizes the surface $\Sigma$ admits a (closed) fundamental polygon $\Pi$, two of whose paired sides are given by the radial lines at angles $0$ and $2\pi/n$. The remaining
\begin{align*}
p=
\begin{cases}
2(\delta_1-1) \quad \textrm{ when }\ \delta_2=0,\\ 
2\delta_1-1 \qquad \textrm{ when }\ \delta_2=1,
\end{cases}
\end{align*}
sides of $\Pi$ are bi-infinite geodesics in $\D$.
The $n$-fold cyclic cover $\widetilde{\Sigma}$ is uniformized by a Fuchsian group $\widetilde{\Gamma}$ which admits a (closed) ideal $m=np-$gon $\widetilde{\Pi}$ as a fundamental domain, and this fundamental domain $\widetilde{\Pi}$ is obtained by gluing $n$ copies of $\Pi$ cyclically around the origin. In particular, $\widetilde{\Pi}$ has ideal vertices at the $n-$th roots of unity (all of which are identified) and it is symmetric under rotation by $2\pi/n$ around the origin. 

We note that $\Gamma=\widetilde{\Gamma}\rtimes \langle M_\omega\rangle$, where $M_\omega(z)=\omega z$, and $\omega:=\exp(2\pi i/n)$.

Due to the $2\pi/n-$rotational symmetry of the construction, the Bowen-Series map $A^{\textrm{BS}}_{\widetilde{\Sigma}}\equiv A^{\textrm{BS}}_{\widetilde{\Gamma}}:\overline{\D}\setminus\Int{\widetilde{\Pi}}\to\overline{\D}$ of $\widetilde{\Gamma}$ equipped with the fundamental domain $\widetilde{\Pi}$ commutes with $M_\omega$. (The map $A^{\textrm{BS}}_{\widetilde{\Gamma}}$ has jump discontinuities at the $n-$th roots of unity, but is continuous otherwise.)
This symmetry allows one to pass to a factor of the above Bowen-Series map on the quotient cone $\D/\langle M_\omega\rangle$. The resulting map is denoted by 
$$
\widehat{A}^{\textrm{BS}}_{\widetilde{\Sigma}}: \left(\overline{\D}\setminus\Int{\widetilde{\Pi}}\right)/\langle M_\omega\rangle\to\overline{\D}/\langle M_\omega\rangle.
$$ 
Let $\xi:\overline{\D}/\langle M_\omega\rangle\to\overline{\D}$ be a uniformization of the cone $\overline{\D}/\langle M_\omega\rangle$ by the closed disk $\overline{\D}$ induced by $z\mapsto z^n$. Then, the \emph{factor Bowen-Series map} associated with $\Sigma$ is defined as
$$
A^{\textrm{fBS}}_{\Sigma}:=\xi\circ \widehat{A}^{\textrm{BS}}_{\widetilde{\Sigma}} \circ\xi^{-1}: \overline{\D}\setminus\Int{\mathcal{H}}\to\overline{\D},
$$ 
where $\mathcal{H}:=\xi(\widetilde{\Pi}/\langle M_\omega\rangle)$. The set $\mathcal{H}$ has $p$ ideal boundary points on~$\mathbb{S}^1$.

By \cite[Proposition~2.5]{MM2}, the factor Bowen-Series map $A^{\textrm{fBS}}_{\Sigma}$ is a piecewise analytic, orientation-preserving, expansive covering map of $\mathbb{S}^1$ of degree  
\begin{align*}
d\equiv d(\Sigma)= m-1 = 1-2n \cdot\chi_{\mathrm{orb}}(\Sigma).
 \end{align*} 
Moreover, when $n\geq 3$, the map $A_\Sigma^{\mathrm{fBS}}$ has $p$ critical points, each of multiplicity $n-1$.
All these critical points are mapped to the same critical value. Finally, the factor Bowen-Series map $A_\Sigma^{\mathrm{fBS}}$ restricts to a self-homeomorphism of order two on~$\partial\mathcal{H}$.

It is easily checked that factor Bowen-Series maps are examples of virtually mateable maps. We also note that factor Bowen-Series maps generalize Bowen-Series maps of punctured sphere Fuchsian groups described in Subsection~\ref{sec-bs}.

According to Proposition~\ref{prop-vmateable}, such maps can be conformally mated with polynomials lying in principal hyperbolic components of appropriate degree. The following considerably stronger version of this mating statement was proved in \cite{LLM24}:

\begin{theorem}\cite[Theorem~1.6]{LLM24}\label{conf_mating_fbs_poly_thm}
Let $P$ be a degree $d$ polynomial with a connected Julia set.
Suppose that $P$ is either 
\begin{itemize}
\item geometrically finite; or 
\item periodically repelling  (i.e., all cycles of $P$ in $\C$ are repelling), finitely renormalizable.
\end{itemize}
Then, $P$ can be conformally mated with any degree $d$ factor Bowen-Series map, and the resulting conformal mating is unique up to M{\"o}bius conjugacy. 
\end{theorem}

The conformal matings of Theorem~\ref{conf_mating_fbs_poly_thm} turn out to be algebraic functions (see \cite[Theorem~14.5,Theorem 15.8]{LLM24}). This algebraic description can be used to construct algebraic correspondences on (possibly nodal) Riemann spheres that capture the full dynamics of the Fuchsian groups (uniformizing genus zero orbifolds in $\mathfrak{F}$) as well as of the polynomials.

\begin{theorem}\cite[Theorem~1.9]{LLM24}\cite[Theorem~B]{MM2}\label{corr_mating_fbs_poly_thm}
Let $\Sigma\in\mathfrak{F}$ with the corresponding Fuchsian group $G$.
Let $P$ be a degree $d(\Sigma)$ polynomial with a connected Julia set which is either
\begin{itemize}
\item geometrically finite; or 
\item periodically repelling, finitely renormalizable.
\end{itemize}
Then there exists a holomorphic correspondence $\mathfrak{C}$ on a (possibly nodal) Riemann sphere which is a mating of $P$ and $G$.
\end{theorem}

\section{Algebraic correspondences uniformizing two genus zero orbifolds}\label{qf_corr_sec}

The passage from a genus zero orbifold group to its factor Bowen-Series map can be thought of as a `forgetful procedure' from an invertible dynamical system to a non-invertible one. However, as explained in the previous section, conformal matings of factor Bowen-Series maps with polynomials can in fact be promoted to algebraic correspondences where the complete dynamical structure of the groups is recovered.

In this section, we will illustrate a new application of this mating framework by establishing a combination theorem for a pair of topologically distinct genus zero orbifolds. The `forgetfulness' mentioned above is key to this construction; indeed, we will show that two factor Bowen-Series maps can be conformally mated (producing a holomorphic map on a subset of the sphere) provided that they have the same degree on $\mathbb{S}^1$, even if the underlying topological surfaces are not homeomorphic. Subsequently, we will give an algebraic description of this mating, which will facilitate the construction of an algebraic correspondence which uniformizes two topologically nonequivalent genus zero orbifolds. The resulting correspondence can be regarded as a generalization of quasi-Fuchsian groups that uniformize a pair of homeomorphic surfaces.  

\subsection{Conformal mating of factor Bowen-Series maps}\label{conf_mat_subsec}

For the rest of this section, let us fix $\Sigma_1,\Sigma_2\in\mathfrak{F}$ (see Subsection~\ref{fbs_sec} for the definition of $\mathfrak{F}$) such that 
\begin{itemize}
\item $d(\Sigma_1)=d(\Sigma_2)$, 
\item $\Sigma_1\ncong\Sigma_2$. 
\end{itemize}
Such examples arise in the following ways.

\begin{enumerate}\label{equality_cases}
\item $\Sigma_1=$ sphere with $\delta_1$ punctures, $\Sigma_2=$ sphere with $\delta_1'$ punctures and an order $\nu\geq 3$ orbifold point such that $\nu(\delta_1'-1)=\delta_1-1$. In this case, $\Sigma_1$ is homeomorphic to the $\nu-$fold cover $\widetilde{\Sigma}_2$ of $\Sigma_2$.

\item $\Sigma_1=$ sphere with $\delta_1$ punctures, $\Sigma_2=$ sphere with $\delta_1'$ punctures, an order $2$ orbifold point, and an order $\nu\geq 3$ orbifold point such that $\nu(2\delta_1'-1)=2\delta_1-2$.

\item $\Sigma_1=$ sphere with $\delta_1$ punctures and an order $2$ orbifold point, $\Sigma_2=$ sphere with $\delta_1'$ punctures, an order $2$ orbifold point, and an order $\nu\geq 3$ orbifold point such that $\nu(2\delta_1'-1)=2\delta_1-1$.

\item $\Sigma_1=$ sphere with $\delta_1$ punctures and an order $\nu_1\geq 3$ orbifold point, $\Sigma_2=$ sphere with $\delta_1'$ punctures and an order $\nu_1'\geq 3$ orbifold point, such that $\nu_1(\delta_1-1)=\nu_1'(\delta_1'-1)$.

\item $\Sigma_1=$ sphere with $\delta_1$ punctures, an order $2$ orbifold point, and an order $\nu_1\geq 3$ orbifold point, $\Sigma_2=$ sphere with $\delta_1'$ punctures, an order $2$ orbifold point, and an order $\nu_1'\geq 3$ orbifold point, such that $\nu_1(2\delta_1-1)=\nu_1'(2\delta_1'-1)$.

\item $\Sigma_1=$ sphere with $\delta_1$ punctures and an order $\nu_1\geq 3$ orbifold point, $\Sigma_2=$ sphere with $\delta_1'$ punctures, an order $2$ orbifold point, and an order $\nu_1'\geq 3$ orbifold point such that $2\nu_1(\delta_1-1)=\nu_1'(2\delta_1'-1)$.
\end{enumerate}

Let $A_1\equiv A^{\textrm{fBS}}_{\Sigma_1}$ and $A_2\equiv A^{\textrm{fBS}}_{\Sigma_2}$ be the factor Bowen-Series maps associated with the surfaces $\Sigma_1$ and $\Sigma_2$. Since each $A_j$, $j\in\{1,2\}$, is an expansive circle covering of degree $d:=d(\Sigma_1)=d(\Sigma_2)$, there exist unique circle homeomorphisms $\mathfrak{g}_j:\mathbb{S}^1\to\mathbb{S}^1$ conjugating $z^d$ to $A_j$ and carrying $1$ to $1$. 

\begin{defn}\label{conf_mating_def}
The maps $A_1:\overline{\D}\setminus\Int{\mathcal{H}_1}\to\overline{\D}$ and $A_2:\overline{\D}\setminus\Int{\mathcal{H}_2}\to\overline{\D}$ are said to be \emph{conformally mateable} if there exist a continuous map $F\colon \mathrm{Dom}(F)\subsetneq\widehat{\C}\to\widehat{\C}$ (called a \emph{conformal mating} of $A_1$ and $A_2$) that is complex-analytic in the interior of $\mathrm{Dom}(F)$ and homeomorphisms $\mathfrak{X}_j:\overline{\D}\to\widehat{\C}$, $j\in\{1,2\}$, conformal on $\D$, satisfying
	\begin{enumerate}[\upshape ({{CM}}-1)]
		\item\label{topo_cond} $\mathfrak{X}_1\left(\overline{\D}\right)\cup \mathfrak{X}_2\left(\overline{\D}\right) = \widehat{\C}$,
		
		\item\label{jordan_cond} $\Lambda:=\mathfrak{X}_1(\mathbb{S}^1)=\mathfrak{X}_2(\mathbb{S}^1)$ is a Jordan curve with $\mathfrak{X}_1(\mathfrak{g}_1(w))=\mathfrak{X}_2(\mathfrak{g}_2(\overline{w}))$ for $w\in\mathbb{S}^1$,
		
		\item\label{dom_cond} $\mathrm{Dom}(F)= \mathfrak{X}_1\left(\overline{\D}\setminus\Int{\mathcal{H}_1}\right)\cup\mathfrak{X}_2\left(\overline{\D}\setminus\Int{\mathcal{H}_2}\right)$, and
		
		\item $\mathfrak{X}_j\circ A_j(z) = F\circ \mathfrak{X}_j(z),\quad \mathrm{for}\ z\in\overline{\D}\setminus\Int{\mathcal{H}_j}$, $j\in\{1,2\}$.
				\end{enumerate}
The maps $\mathfrak{X}_j$, $j\in\{1,2\}$, are called \emph{mating conjugacies} associated with the conformal mating $F$ of $A_1$ and $A_2$. 
We say that the mating of $A_1$ and $A_2$ is unique if $F$ is unique up to M{\"o}bius conjugation.
\end{defn}

If a mating $F$ exists, the point $\mathfrak{X}_1(1)=\mathfrak{X}_2(1)$ is a marked fixed point of $F$ on $\Lambda$.

\begin{prop}\label{conf_mating_exists_prop}
The maps $A_1$ and $A_2$ are conformally mateable. Moreover, the conformal mating is unique. 
\end{prop}
\begin{proof}
The existence of the desired conformal mating is a consequence of \cite[Theorem~5.2]{LMMN}. We include the mating construction for completeness and future reference. 

Let $P_0(z)$ be the map $z\mapsto z^d$, where $d$ is the common degree of $A_1, A_2$ on $\mathbb{S}^1$.
By the proof of \cite[Lemma~3.4]{MM2}, each $A_j$, $j\in\{1,2\}$, admits a Markov partition satisfying conditions (4.1) and (4.2) of \cite[Theorem~5.2]{LMMN}. Moreover, each periodic break-point of its piecewise analytic definition is symmetrically parabolic (cf. \cite[Definition~4.6, Remark~4.7]{LMMN}). 
By \cite[Theorem~4.13]{LMMN}, the circle homeomorphisms $\mathfrak{g}_j$, that conjugate $P_0$ to $A_j$, $j\in \{1,2\}$, extend continuously to David homeomorphisms~of~$\D$. 

Let $\eta(z):=1/z$ and $\widetilde{\mathfrak{g}}_2= \mathfrak{g}_2\circ \eta:\widehat{\C}\setminus\D\to\overline{\D}$. We first define the \emph{topological mating} of $A_1$ and $A_2$ as follows:
\begin{align*}
\widetilde{F}:= \begin{cases} \mathfrak{g}_1^{-1}\circ A_1\circ \mathfrak{g}_1,  & \textrm{on}\,\, \overline{\D}\ \setminus\ \mathfrak{g}_1^{-1}(\Int{\cH}_1)\\
    \widetilde{\mathfrak{g}_2}^{-1}\circ A_2\circ \widetilde{\mathfrak{g}}_2, &\textrm{on}\,\, \overline{\D^*}\ \setminus\ \widetilde{\mathfrak{g}_2}^{-1}(\Int{\cH}_2), 
\end{cases}
\end{align*}
where $\D^*:=\widehat{\C}\setminus\overline{\D}$.
The two definitions agree on $\mathbb S^1$. The domain of definition $\mathrm{Dom}(\widetilde{F})$ of the topological mating is $\widehat{\C}\setminus\left(\mathfrak{g}_1^{-1}(\Int{\cH}_1)\cup\widetilde{\mathfrak{g}_2}^{-1}(\Int{\cH}_2)\right)$ (see Figure~\ref{mating_domain_model_fig}).
\begin{figure}[h!]
\captionsetup{width=0.96\linewidth}
\begin{tikzpicture}
\node[anchor=south west,inner sep=0] at (0,0) {\includegraphics[width=0.49\textwidth]{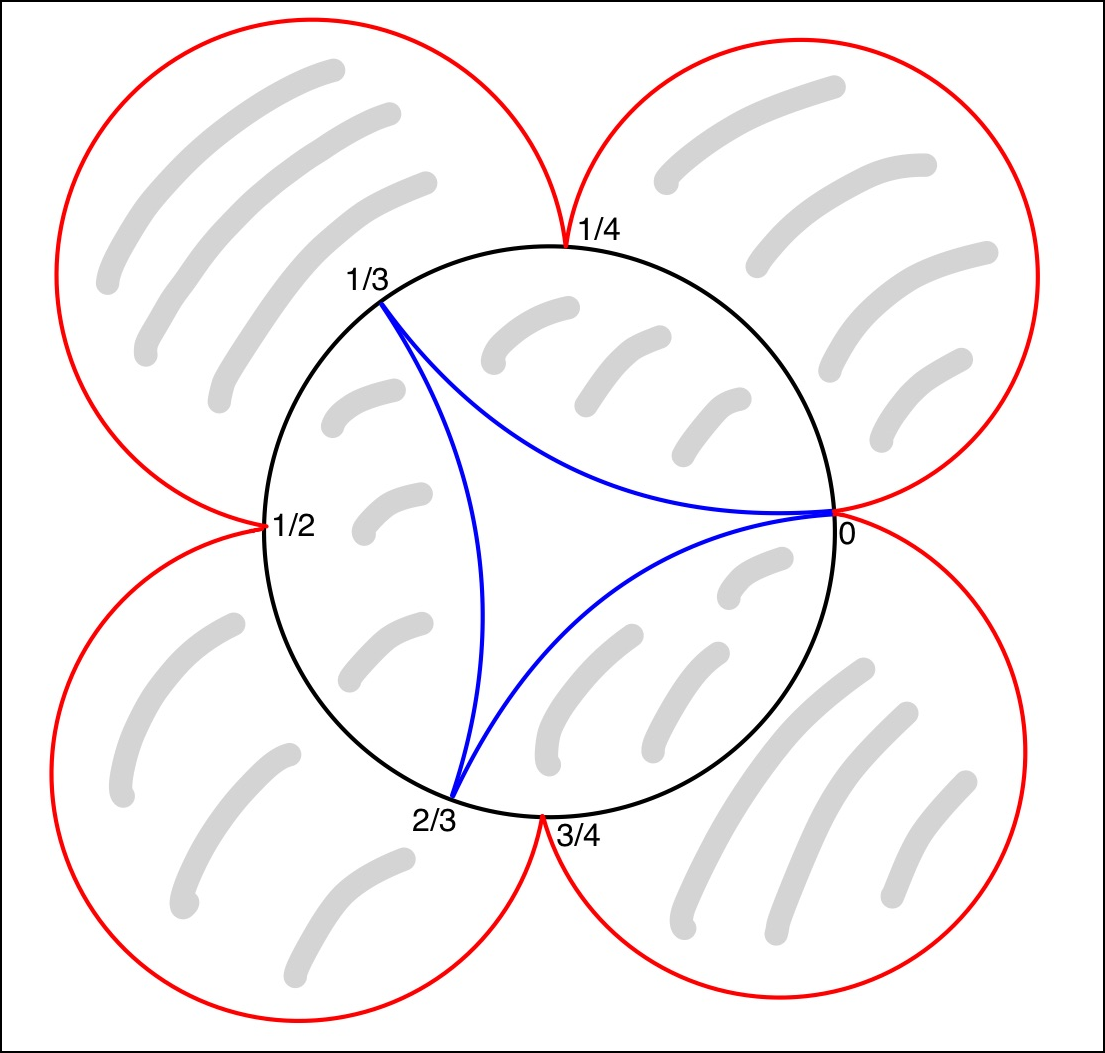}}; 
\node[anchor=south west,inner sep=0] at (6.4,0) {\includegraphics[width=0.48\textwidth]{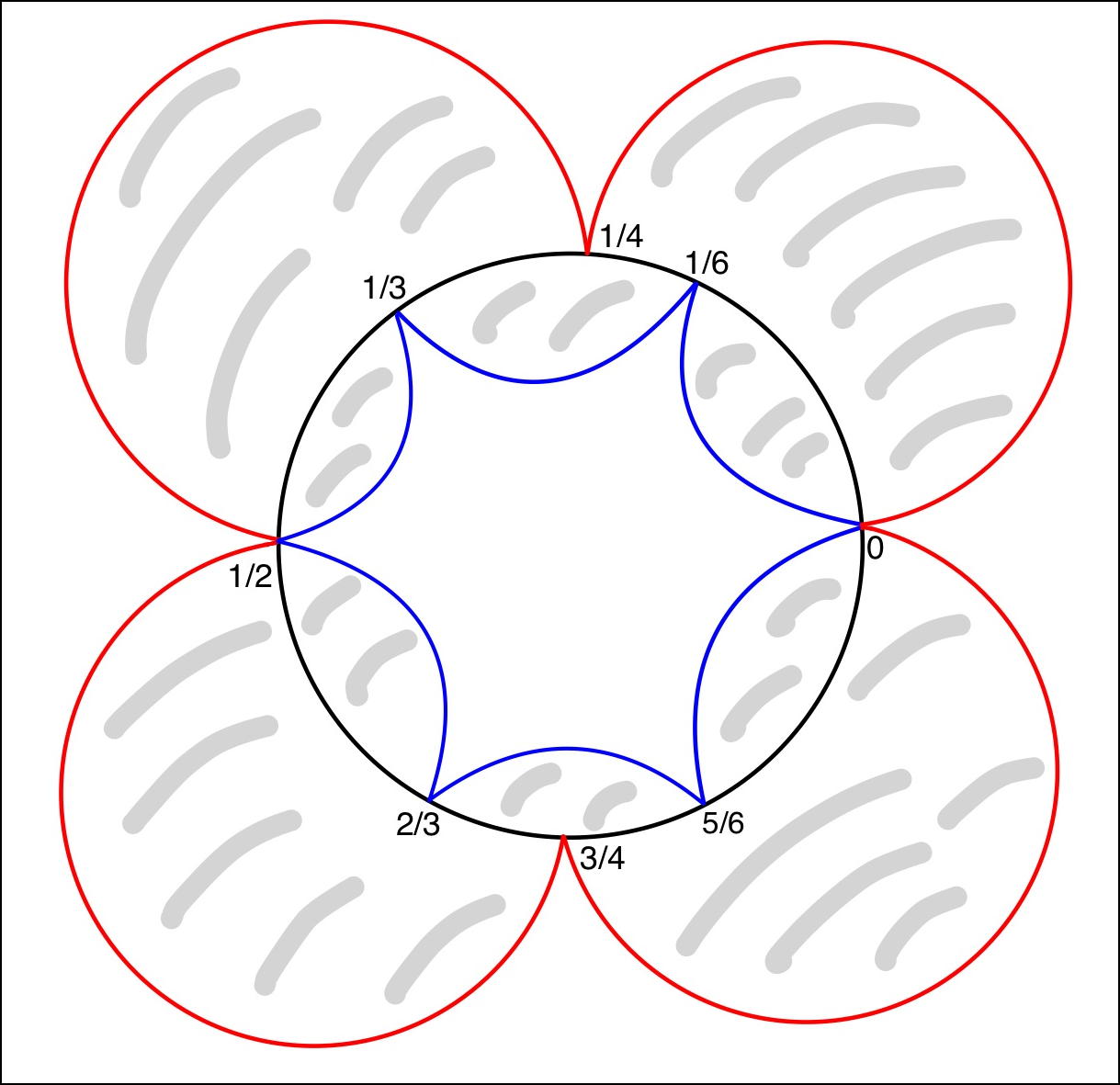}}; 
\node at (3.25,2.96) {\begin{tiny}$\mathfrak{g}_1^{-1}(\cH_1)$\end{tiny}};
\node at (5.64,3.02) {\begin{tiny}$\widetilde{\mathfrak{g}}_2^{-1}(\cH_2)$\end{tiny}};
\node at (9.6,2.96) {\begin{tiny}$\mathfrak{g}_1^{-1}(\cH_1)$\end{tiny}};
\node at (11.94,3.02) {\begin{tiny}$\widetilde{\mathfrak{g}}_2^{-1}(\cH_2)$\end{tiny}};
\end{tikzpicture}
\caption{Left: The shaded region is the domain of definition of the topological mating of a factor Bowen-Series for a sphere with $3$ punctures and an order $3$ orbifold point, and a factor Bowen-Series for a sphere with $2$ punctures, an order $2$ orbifold point and an order $4$ orbifold point. The interior of the domain of definition is a simply connected domain. Right: The shaded region is the domain of definition of the topological mating of a factor Bowen-Series for a sphere with $4$ punctures and an order $4$ orbifold point, and a factor Bowen-Series for a sphere with $3$ punctures and an order $6$ orbifold point. The interior of the domain of definition has two simply connected components.}
\label{mating_domain_model_fig}
\end{figure}
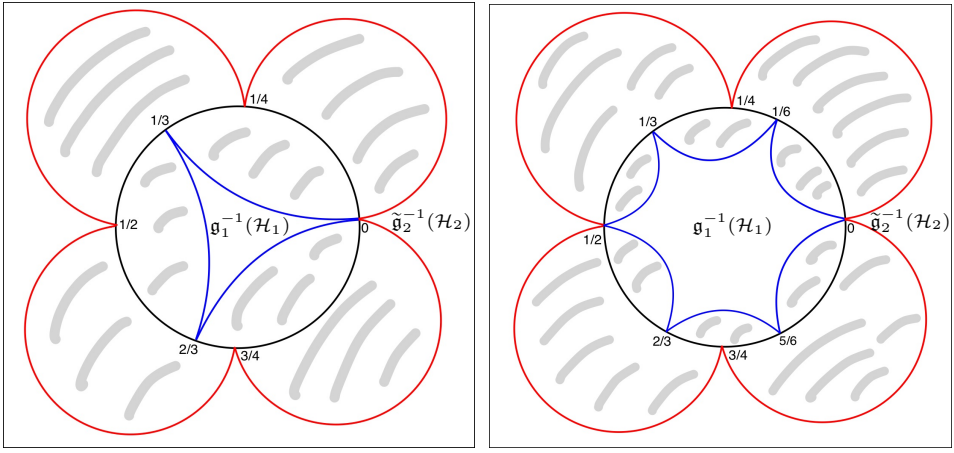

Next, we define an $\widetilde{F}$-invariant David (Beltrami) coefficient $\mu$ on $\widehat{\C}$ as follows (see \cite[\S 2]{LMMN} for background on David homeomorphisms and David coefficients). In $\D$, we let $\mu$ be the pullback of the standard complex structure under the David homeomorphism $\mathfrak{g}_1$. In $\widehat{\C}\setminus \overline{\D}$, we let $\mu$ be the pullback of the standard complex structure under map $\widetilde{\mathfrak{g}}_2$. Then, $\mu$ is a David coefficient on $\widehat{\C}$. Since $A_1, A_2$ are holomorphic, it follows that $\mu$ is $\widetilde{F}$-invariant.

By the David Integrability Theorem (see \cite{Dav88}, \cite[Theorem~20.6.2, p. 578]{AIM09}), there exists a David homeomorphism $H$ of $\widehat{\C}$ that solves the Beltrami equation with coefficient $\mu$. Consider the map 
$$
F:=H\circ\widetilde{F}\circ H^{-1}:\mathrm{Dom}(F):=H\left(\mathrm{Dom}(\widetilde{F})\right)\to\widehat{\C}.
$$
By \cite[Theorem~2.2]{LMMN}, the maps $\mathfrak{X}_1:=H\circ\mathfrak{g}_1^{-1}:\D\to\widehat{\C}$ and $\mathfrak{X}_2:=H\circ\widetilde{\mathfrak{g}}_2^{-1}:\D\to\widehat{\C}$ are conformal.
Hence, $F$ is holomorphic on $\Int{(\mathrm{Dom}(F))}\setminus\Lambda$, where $\Lambda:=H(\mathbb S^1)$.
Further, $F$ extends to a homeomorphism in a neighborhood of $\Lambda$, pinched at finitely many points, such that this extension is conformal outside $\Lambda$. Since David circles are locally conformally removable (cf. \cite[Theorem 2.8]{LMMN}), it follows that the above local extensions are conformal. Hence, $F$ is holomorphic on $\Int{\mathrm{Dom}(F)}$.
It is readily checked that $\mathfrak{X}_1, \mathfrak{X}_2$ are the desired mating conjugacies. 

Uniqueness of the mating follows from conformal removability of the curve~$\Lambda$.
\end{proof}

To give an explicit description of the mating $F$ of $A_1$ and $A_2$, we need the following definition.

\begin{defn}\label{b_inv_def}
Let $\{\Omega_1,\cdots,\Omega_k\}$ be a disjoint collection of proper simply connected sub-domains of $\widehat{\C}$ such that $\Int{\overline{\Omega_j}}=\Omega_j$, $j\in\{1,\cdots, k\}$, and let $\cD:=\displaystyle\bigsqcup_{j=1}^k\Omega_j$. Further, let $\mathfrak{S}\subset\partial\cD$ be a finite set such that $\partial^0\cD:=\partial\cD\setminus\mathfrak{S}$ is a finite union of disjoint non-singular real-analytic curves.

\noindent The set $\cD$ is called an \emph{inversive multi-domain} if it admits a continuous map $S:\overline{\cD}\to\widehat{\C}$ satisfying the properties:
\begin{enumerate}[\upshape ({{I}}-1)]
\item\label{mero_cond} $S$ is meromorphic on $\cD$,
\item\label{permute_cond} $S(\partial\Omega_j)=\partial\Omega_{j'}$, for some $j'\in\{1,\cdots,k\}$, and
\item\label{inv_cond} $S:\partial\cD\to\partial\cD$ is an orientation-reversing involution preserving $\mathfrak{S}$.
\end{enumerate}
The map $S$ is called a \emph{B-involution} of the inversive multi-domain $\cD$.

\noindent When $k=1$, the domain $\cD$ is called an \emph{inversive domain}.
\end{defn}

\begin{prop}\label{conf_mating_b_inv_prop}
The conformal mating $F$ of $A_1$ and $A_2$ is a B-involution of an inversive multi-domain $\cD$. Further, if $\gcd(p_1,p_2)=1$, then $\cD$ is connected.
\end{prop}
\begin{proof}
Let us denote the ideal boundary points of $\cH_j$ on $\mathbb{S}^1$ by $\mathscr{I}_j$, $j\in\{1,2\}$.
We set 
$$
\cD:=\Int{\mathrm{Dom}(F)}, \quad \mathfrak{S}_j:=\mathfrak{X}_j(\mathscr{I}_j),\ j\in\{1,2\} \quad \textrm{and} \quad \mathfrak{S}:=\mathfrak{S}_1\cup\mathfrak{S}_2.
$$
The facts that $\partial\cH_j\setminus\mathscr{I}_j$ consists of finitely many disjoint non-singular real-analytic curves and that $\mathfrak{X}_j$ is conformal on $\D$, $j\in\{1,2\}$, imply that $\partial\cD\setminus\mathfrak{S}$ is a finite union of disjoint non-singular real-analytic curves. Further, the meromorphic map $F:\overline{\cD}\to\widehat{\C}$ preserves the set $\mathfrak{S}$. 

Note that the homeomorphism $\mathfrak{g}_j$ pulls back the set $\mathscr{I}_j$ to the $p_j-$th roots of unity, $j\in\{1,2\}$ (cf. \cite[\S 4.1]{MM2}). It now follows from the definition of $F$ that $\mathfrak{X}_1(\mathscr{I}_1)$ and $\mathfrak{X}_2(\mathscr{I}_2)$ intersect precisely at the $r:=\gcd(p_1,p_2)$ points 
$$
\{\mathfrak{X}_1(\mathfrak{g}_1(w))=\mathfrak{X}_2(\mathfrak{g}_2(\overline{w})): w\ \textrm{ is an}\ r\textrm{-th root of unity}\}.
$$
One of these points is $\mathfrak{X}_1(1)=\mathfrak{X}_2(1)$. This implies that $\cD$ is a disjoint union of proper simply connected domains $\Omega_j\subsetneq\widehat{\C}$ with $\Int{\overline{\Omega_j}}=\Omega_j$, $j\in\{1,\cdots, k\}$. 

The fact that each $A_j$ induces an orientation-reversing self-homeomorphism of order two on $\partial\cH_j$, $j\in\{1,2\}$, implies that $F:\partial\cD\to\partial\cD$ is an orientation-reversing involution preserving $\mathfrak{S}$. It also follows from the above discussion that $F$ carries $\partial\Omega_j$ to $\partial\Omega_{k+1-j}$, $j\in\{1,\cdots,k\}$ (after possibly renumbering the $\Omega_j$s).

Finally, if $\gcd(p_1,p_2)=1$, then $\mathfrak{X}_1(\partial\cH_1)\cap\mathfrak{X}_2(\partial\cH_2)$ is a singleton, and hence $\cD$ is connected. 
\end{proof}

\subsection{Correspondence uniformizing a pair of genus zero orbifolds}\label{sec-proof_mainthm}

We refer the reader to \cite{BP01} for the notion of regular and limit sets for holomorphic correspondences.

\begin{proof}[Proof of Theorem~\ref{simult_unif_corr_thm}]
Let $\eta(z)=1/z$, and $\kappa:\{1,\cdots,k\}\to\{1,\cdots,k\},\ \kappa(j)= k+1-j$.
By Proposition~\ref{conf_mating_b_inv_prop} and \cite[\S 16]{LLM24}, there exist Jordan domains $\mathfrak{D}_j$ and rational maps $R_j$, $j\in\{1,\cdots,k\}$, such that the following hold.
\begin{enumerate}
\item $\eta:\mathfrak{D}_j\to\widehat{\C}\setminus\overline{\mathfrak{D}_{\kappa(j)}}$ is a homeomorphism.
\item $\partial\mathfrak{D}_j$ is a piecewise non-singular real-analytic curve.
\item  $R_j:\mathfrak{D}_j\to\Omega_j$ is a conformal isomorphism.
\item $F\vert_{\Omega_j}\equiv R_{\kappa(j)}\circ\eta\circ(R_j\vert_{\mathfrak{D}_j})^{-1}$.
\end{enumerate}
For notational convenience, we denote the domain of $R_j$ by $\widehat{\C}_j$. Consider the disjoint union 
$$
\mathfrak{U}\ :=\ \bigsqcup_{j=1}^k \widehat{\C}_{j}
$$
and define the maps
$$
\pmb{R}:\ \mathfrak{U}\longrightarrow \widehat{\C},\quad (z,j)\mapsto R_{j}(z),
$$
and
$$
\pmb{\eta}\ : \mathfrak{U}\longrightarrow \mathfrak{U},\quad (z,j)\mapsto (\eta(z),\kappa(j)).
$$
By construction, $\pmb{R}$ is a branched covering of degree $d+1$, and $\pmb{\eta}$ is a homeomorphism. 

Following \cite[\S 5.2]{MM2} (cf. \cite[\S 17]{LLM24}), one can lift the conformal mating $F$ by the degree $d+1$ branched cover $\pmb{R}:\mathfrak{U}\to\widehat{\C}$ to obtain a bi-degree $d$:$d$ correspondence $\mathfrak{C}^{\circledast}$ on $\mathfrak{U}$. This correspondence can be written explicitly as follows:
\begin{equation}
\Big\{(\mathfrak{u}_1,\mathfrak{u}_2)\in \mathfrak{C}^{\circledast}\subset\mathfrak{U}\times\mathfrak{U}: \frac{\pmb{R}(\mathfrak{u}_2)-\pmb{R}(\pmb{\eta}(\mathfrak{u}_1))}{\mathfrak{u}_2-\pmb{\eta}(\mathfrak{u}_1)}=0\Big\}.
\label{holo_corr_eqn}
\end{equation}
We then pass to the quotient 
$$
\mathfrak{W}\ :=\ \faktor{\mathfrak{U}}{\sim},
$$
where $\sim$ is the finite equivalence relation defined as
\begin{center}
For $z\in\partial\mathfrak{D}_i\subset\widehat{\C}_i$ and $w\in\partial\mathfrak{D}_j\subset\widehat{\C}_j,\ i\neq j$
\smallskip

$(z,i)\sim (w,j)\iff R_i(z)=R_j(w)$.
\end{center}
The space $\mathfrak{W}$ can be viewed as a compact, (possibly) noded Riemann surface.
It is easily checked that the maps $\pmb{R}, \pmb{\eta}$ descend to $\mathfrak{W}$, defining a bi-degree $d$:$d$ correspondence $\mathfrak{C}$ on $\mathfrak{W}$ (see \cite[Lemma~5.11]{MM2}).

We now set 
$$
\cT_j:=\pmb{R}^{-1}(\mathfrak{X}_j(\D))\subset \mathfrak{W},\ j\in\{1,2\}.
$$
The arguments of \cite[\S 5.1.2]{MM2} show that $\cT_j$ is the disjoint union of $p_j$ simply connected domains, each of which is mapped by $\pmb{R}$ onto $\mathfrak{X}_j(\D)$ with degree $n_j$ (see Figure~\ref{hecke_punc_sphere_conf_mating_fig}). Moreover, by \cite[Proposition~5.13]{MM2} (also see \cite[\S 5.1.2, \S 5.1.3]{MM2}), the forward branches of the correspondence $\mathfrak{C}$ act on $\cT_j$ by conformal automorphisms such that the group $G_j$ generated by these conformal automorphisms act properly discontinuously on $\cT_j$ with $\cT_j/G_j\cong_{\textrm{conf.}}\Sigma_j$. 

Finally, it readily follows from the dynamics of $\mathfrak{C}$ that the regular set $\Omega(\mathfrak{C})$ of $\mathfrak{C}$ is given by $\cT_1\sqcup\cT_2$. Hence, we conclude that the quotient $\Omega(\mathfrak{C})/\mathfrak{C}$ is biholomorphic to the disjoint union of $\Sigma_1$ and $\Sigma_2$.
\end{proof}

\section{A Teichm{\"u}ller space for punctured spheres}\label{bers_sec}

We will now describe parameter space consequences of the combination procedure explicated in Section~\ref{qf_corr_sec}. More precisely, we will consider a collection of algebraic correspondences such that each correspondence uniformizes a given rigid orbifold (i.e., an orbifold admitting only one complex structure) and a sphere with a given number of punctures. As the punctured sphere varies over its Teichm{\"u}ller space, we will obtain a copy of the Teichm{\"u}ller space of punctured spheres in the space of algebraic correspondences.
We now proceed to formalize this construction.

\subsection{Hecke orbifold and punctured spheres}
Let $\Gamma_1\cong\Z/2\Z*\Z/2n\Z$ be the Fuchsian group such that $\D/\Gamma_1$ is the Hecke orbifold $\Sigma_1$; i.e., the genus zero orbifold with one puncture, an order $2$ orbifold point and an order $2n$ orbifold point, for $n\geq 2$. Let $A_1:\overline{\D}\setminus\Int{\mathcal{H}_1}\to\overline{\D}$ be the factor Bowen-Series map of $\Sigma_1$. The map $A_1$ restricts to a degree $2n-1$ covering of $\mathbb{S}^1$, and has a unique critical point, of multiplicity $2n-1$ (see Figure~\ref{hecke_factor_bs_fig}). 
\begin{figure}[ht]
\captionsetup{width=0.98\linewidth}
\begin{tikzpicture}
\node[anchor=south west,inner sep=0] at (0.5,0) {\includegraphics[width=0.45\linewidth]{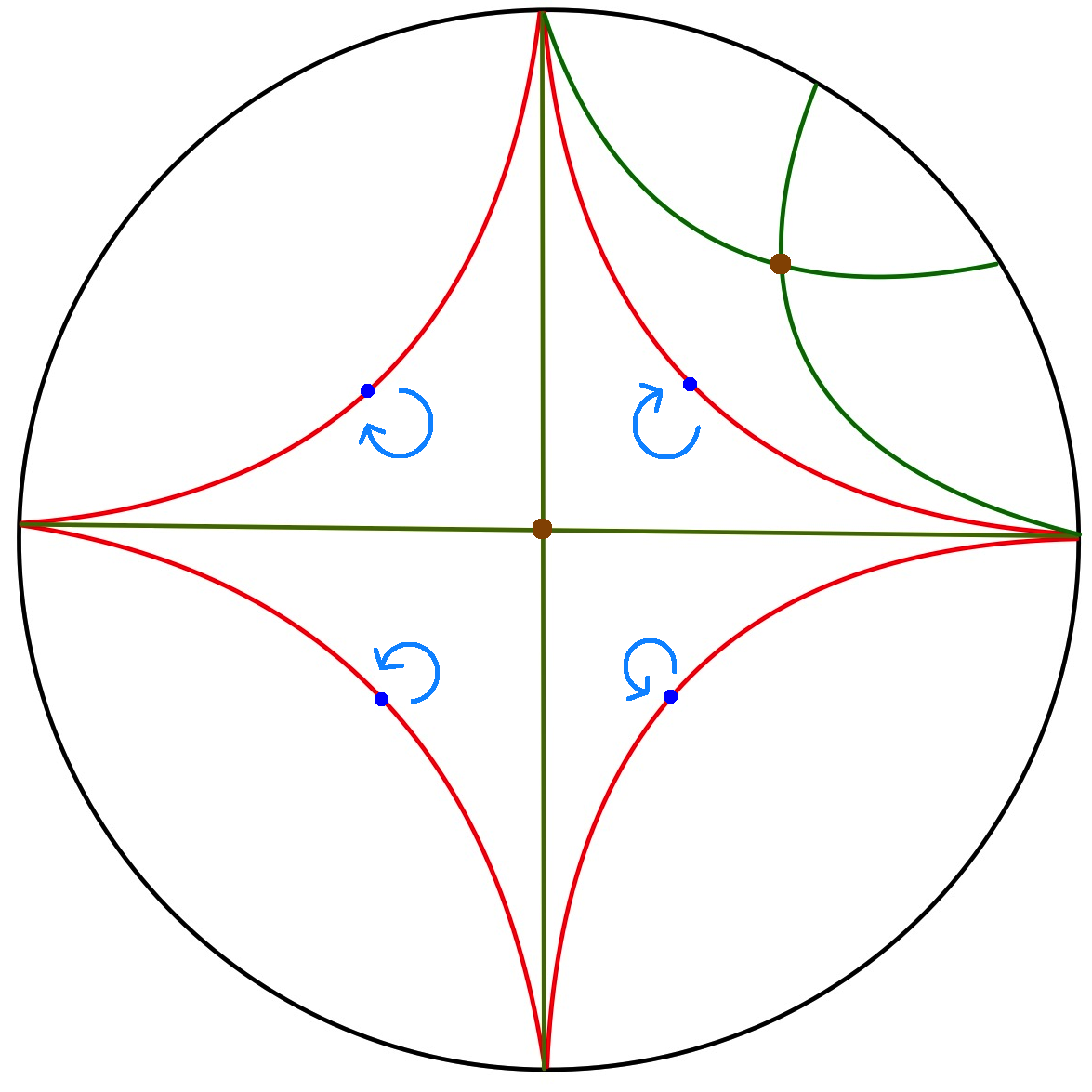}};
\node[anchor=south west,inner sep=0] at (7,0) {\includegraphics[width=0.45\linewidth]{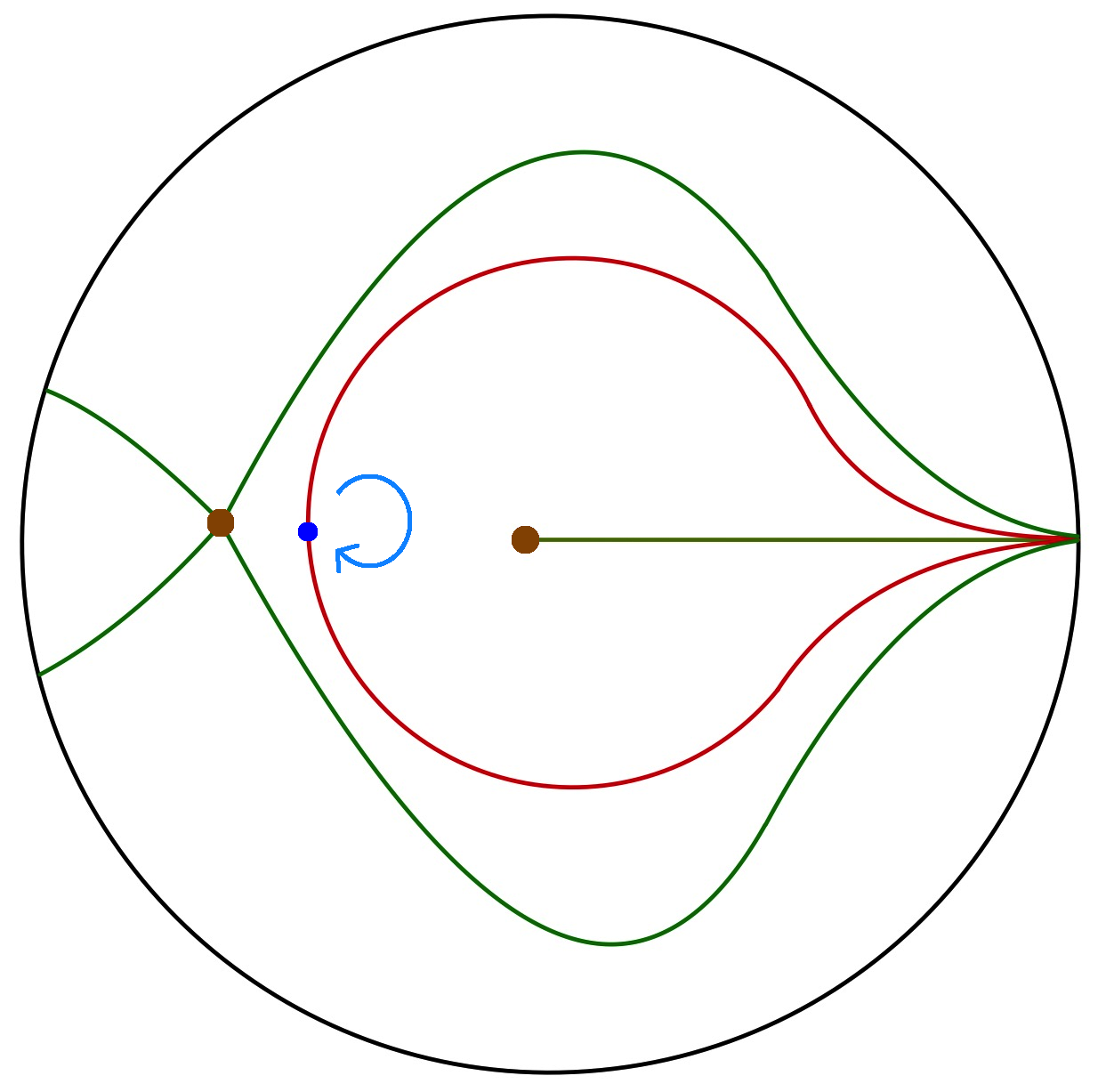}};
\node at (5.08,4.66) {\begin{small}$g_{1}$\end{small}};
\node at (2,4.5) {\begin{small}$g_{2}$\end{small}};
\node at (2,1.5) {\begin{small}$g_{3}$\end{small}};
\node at (4.8,1.5) {\begin{small}$g_{4}$\end{small}};
\node at (3.72,3.2) {\begin{small}$g_{1}$\end{small}};
\node at (2.9,3.2) {\begin{small}$g_{2}$\end{small}};
\node at (2.9,2.5) {\begin{small}$g_{3}$\end{small}};
\node at (3.72,2.5) {\begin{small}$g_{4}$\end{small}};
\node at (3.12,4) {\begin{small}$\Pi_1$\end{small}};
\node at (10,2.2) {$\mathcal{H}_1$};
\end{tikzpicture}
\caption{For the Hecke surface $\Sigma_1$ with an order $4$ orbifold point, the cyclic cover $\widetilde{\Sigma_1}$ is a sphere with one puncture and four order two orbifold points. Left: The preferred fundamental domain $\Pi_1$ and the action of the associated Bowen-Series map $A^{\textrm{BS}}_{\widetilde{\Sigma_1}}$ for $\widetilde{\Sigma_1}$ is shown.
The Bowen-Series map $A^{\textrm{BS}}_{\widetilde{\Sigma_1}}$ commutes with rotation by $\pi/2$. The vertical and horizontal radial lines in $\D$ and their pre-images under $g_{1}$ are displayed in green. 
Right: Depicted is the factor Bowen-Series map $A_1:=A_{\Sigma_1}^{\mathrm{fBS}}:\overline{\D}\setminus\Int{\mathcal{H}_1}\to\overline{\D}$, where $\mathcal{H}_1$ (which is an ideal monogon) is the image of $\Pi_1$ under the projection map $\D\to\D/\langle\zeta\mapsto i\zeta\rangle$. The map $A_1$ has a unique critical point of multiplicity three at the valence four vertex of the green graph.}
\label{hecke_factor_bs_fig}
\end{figure}
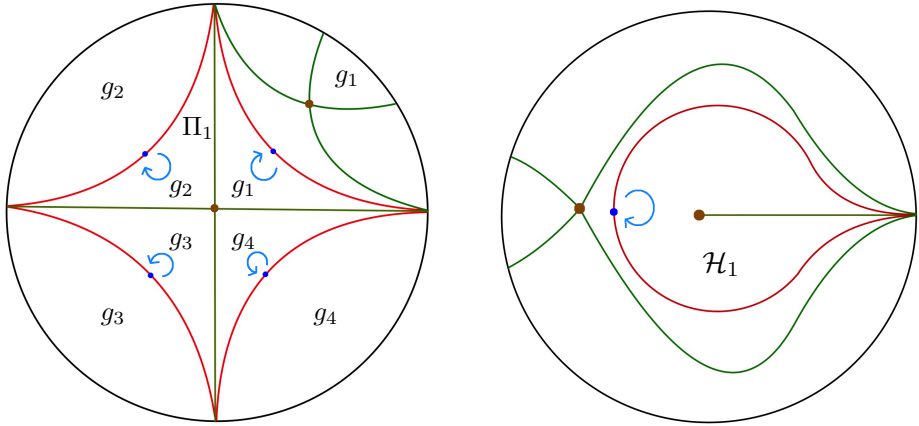

Further, let $\Gamma_2\in\mathrm{Teich}(S_{0,n+1})$. We equip $\Gamma_2$ with the fundamental domain $\Pi_2$ described in Section~\ref{sec-bs}, and set $\cH_2:=\Pi_2$.
Note that the factor Bowen-Series map of $\Gamma_2$ is the usual Bowen-Series map $A_2:\overline{\D}\setminus\Int{\cH_2}\to\overline{\D}$.
We denote the factor Bowen-Series map of any marked group $\Gamma\in\mathrm{Teich}(\Gamma_2)\equiv\mathrm{Teich}(S_{0,n+1})$ by $A_\Gamma:\overline{\D}\setminus\Int{\cH_\Gamma}\to\overline{\D}$. The map $A_\Gamma$ is a piecewise M{\"o}bius degree $2n-1$ circle covering (see Figure~\ref{punc_sphere_bs_fig}). 
\begin{figure}[ht]
\captionsetup{width=0.96\linewidth}
	\begin{tikzpicture}
		\node[anchor=south west,inner sep=0] at (0.5,0) {\includegraphics[width=0.45\linewidth]{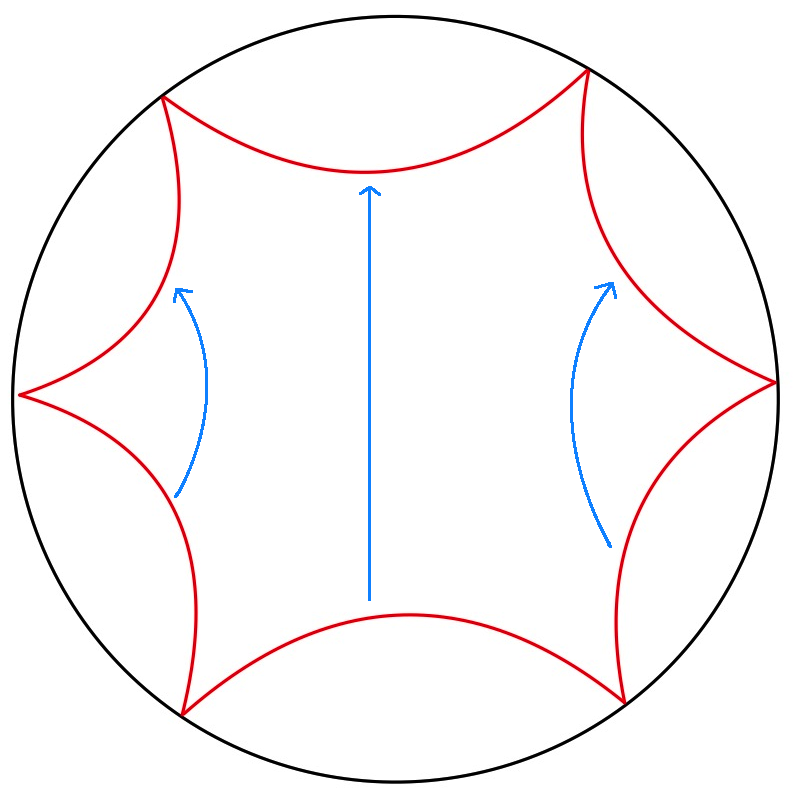}};
		\node at (3.8,3.6) {$\cH_\Gamma$};
		\node at (4.9,2.8) {$g_1$};
		\node at (2.84,2.8) {$g_2$};
		\node at (1.66,2.8) {$g_3$};
		\node at (5.4,4.08) {\begin{small}$g_{1}^{-1}$\end{small}};
		\node at (5.4,1.6) {\begin{small}$g_{1}$\end{small}};
		\node at (3.3,5.25) {\begin{small}$g_{2}^{-1}$\end{small}};
		\node at (3.5,0.6) {\begin{small}$g_{2}$\end{small}};
		\node at (1.24,4) {\begin{small}$g_{3}^{-1}$\end{small}};
		\node at (1.32,1.8) {\begin{small}$g_{3}$\end{small}};	
	\end{tikzpicture}
	\caption{Pictured is the preferred fundamental domain $\cH_\Gamma$ and the action of the associated Bowen-Series map $A_\Gamma:\overline{\D}\setminus\Int{\cH_\Gamma}\to\overline{\D}$ for a four times punctured sphere group.}
	\label{punc_sphere_bs_fig}
\end{figure}
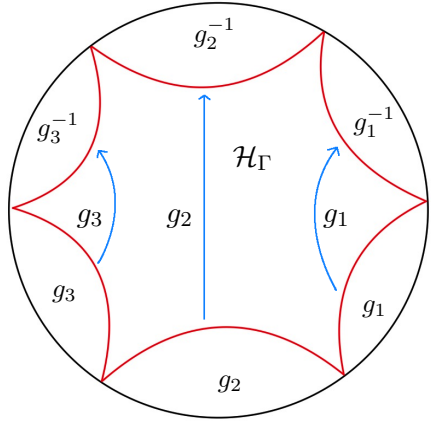
We note that the Teichm{\"u}ller space of $\Sigma_2$ has complex dimension $n-2$.

\subsection{Conformal matings and associated correspondences}

By Proposition~\ref{conf_mating_exists_prop}, the maps $A_1$ and $A_\Gamma$ are conformally mateable. By Proposition~\ref{conf_mating_b_inv_prop}, the conformal mating $F:\overline{\cD}\to\widehat{\C}$ of $A_1$ and $A_\Gamma$ is a B-involution, where $\mathcal{D}$ is a simply connected inversive domain (see Figure~\ref{hecke_punc_sphere_top_mating_fig} for the domain of the topological mating between $A_1$ and $A_\Gamma$). The conformal mating $F$ is unique up to M{\"o}bius conjugacy.
It follows from the construction of $F$ that $\partial\cD$ is homeomorphic to a wedge of two circles with the unique cut-point being $\pmb{x}:=\mathfrak{X}_1(1)=\mathfrak{X}_\Gamma(1)$, where $\mathfrak{X}_1, \mathfrak{X}_\Gamma$ are the mating conjugacies (see Figure~\ref{hecke_punc_sphere_conf_mating_fig}). 
\begin{figure}[h!]
\captionsetup{width=0.96\linewidth}
\begin{tikzpicture}
\node[anchor=south west,inner sep=0] at (0,0) {\includegraphics[width=0.8\textwidth]{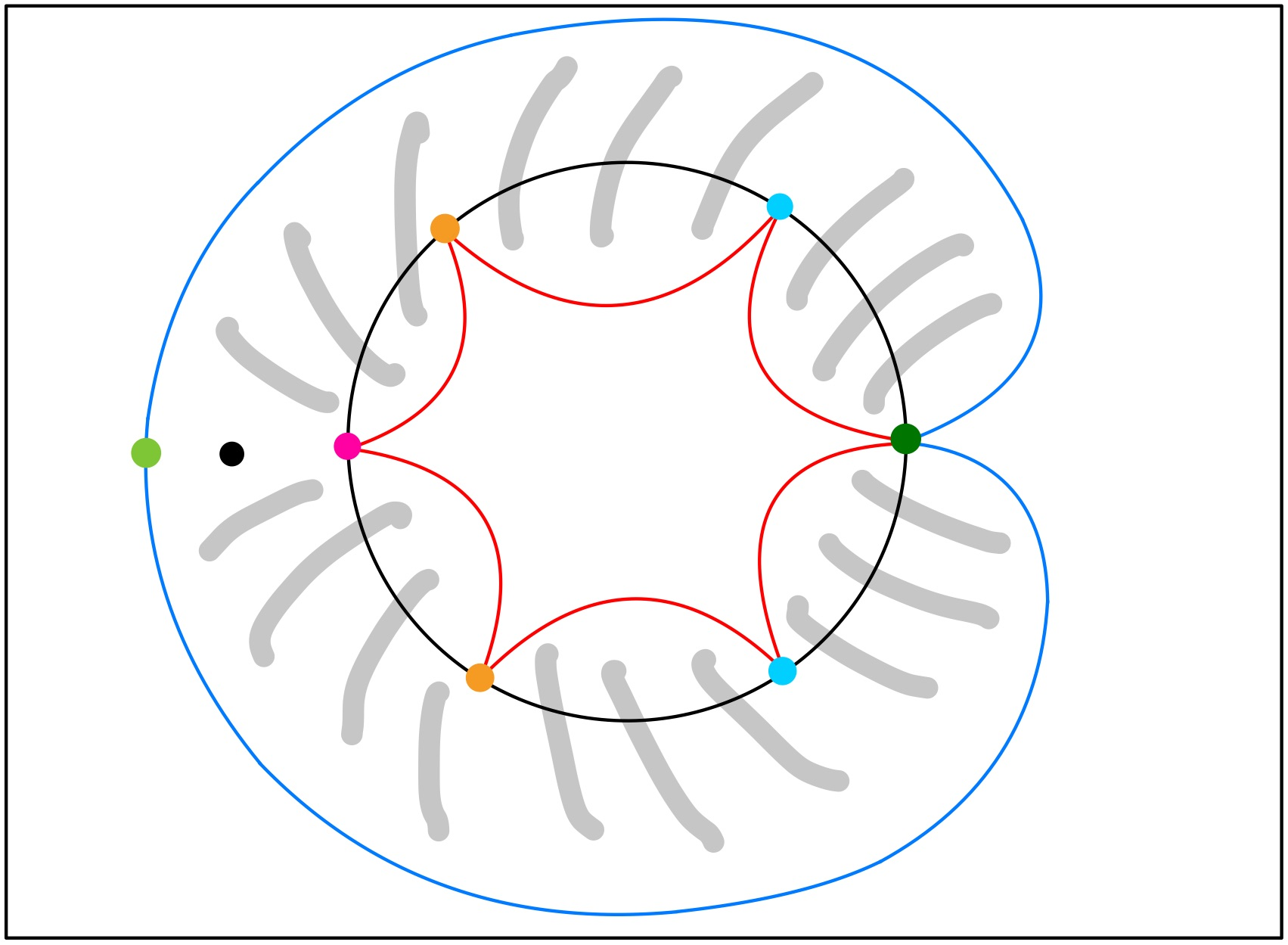}};
\node at (4.9,4) {$\mathfrak{g}_\Gamma^{-1}(\cH_\Gamma)$};
\node at (9,0.8) {$\widetilde{\mathfrak{g}}_1^{-1}(\cH_1)$};
\end{tikzpicture}
\caption{The shaded region is the domain of the topological mating between the factor Bowen-Series map $A_1$ of the $(2,6,\infty)$ genus zero orbifold and a Bowen-Series map $A_\Gamma$ of a four times punctured sphere group $\Gamma$. Here $\mathfrak{g}_\Gamma$ (respectively, $\widetilde{\mathfrak{g}}_1$) is a David homeomorphism from $\overline{\D}$ (respectively, from $\widehat{\C}\setminus\D$) onto $\overline{\D}$ that conjugates $z^{5}$ to $A_\Gamma$ (respectively, to $A_1$) on~$\mathbb{S}^1$. The unique critical point and some of the fixed points and $2$-cycles of the mating are marked.}
\label{hecke_punc_sphere_top_mating_fig}
\end{figure}

Let $R$ be a degree $2n$ rational map and $\mathfrak{D}$ be a Jordan domain such that
\begin{enumerate}[\upshape ({{R}}-1)]
\item\label{eta_inv_cond} $\eta(\mathfrak{D})=\widehat{\C}\setminus\overline{\mathfrak{D}}$, $\pm 1\in\partial\mathfrak{D}$,
\item\label{univ_cond} $R:\mathfrak{D}\to\mathcal{D}$ is a conformal isomorphism, and
\item\label{alg_formula} $F\vert_{\overline{\mathcal{D}}}\equiv R\circ\eta\circ (R\vert_{\overline{\mathfrak{D}}})^{-1}$
\end{enumerate}
(cf. \cite[Lemma~14.3]{LLM24}.) The rational map $R$ is unique in the following sense:
\begin{enumerate}[\upshape ({{U}}-1)]
\item\label{pre_comp} Given a conformal mating $F$ of $A_1$ and $A_\Gamma$, if there are two pairs $(R_j,\mathfrak{D}_j)$, $j\in\{1,2\}$, satisfying the above properties, then there exists a M{\"o}bius map $N$ commuting with $\eta$ such that $N(\mathfrak{D}_1)=\mathfrak{D}_2$ and $R_1=R_2\circ N$ (see the proof of \cite[Proposition~6.1]{MM2}).
\item\label{post_comp} Conjugating $F$ by a M{\"o}bius map $M$ amounts to post-composing $R$ with the map $M$.
\end{enumerate}

By Theorem~\ref{simult_unif_corr_thm}, the algebraic correspondence $\mathfrak{C}_\Gamma$ on the Riemann sphere defined as 
$$
\{(\mathfrak{u}_1,\mathfrak{u}_2)\in\mathfrak{C}_\Gamma\subset\widehat{\C}\times\widehat{\C}: \frac{R(\mathfrak{u}_2)-R(\eta(\mathfrak{u}_1))}{\mathfrak{u}_2-\eta(\mathfrak{u}_1)}=0\}
$$
simultaneously uniformizes the Hecke orbifold $\Sigma_1$ and the marked $(n+1)$-times punctured sphere $\D/\Gamma$. 

\subsection{Explicit description of the correspondences}\label{rat_norm_subsec}

We will now use normalizations~\ref{pre_comp} and~\ref{post_comp} to give an explicit formula for the rational map $R$. Let us denote the centralizer of $\eta(z)=1/z$ in $\mathrm{PSL}_2(\C)$ by $C(\eta)$.

Note that $A_1$ has a unique critical point. This critical point has multiplicity $2n-1$, and has $0$ as its associated critical value. Hence, the conformal mating $F$ also has a unique critical point $\pmb{c}$, of multiplicity $2n-1$. We pre-compose $R$ with $N\in C(\eta)$ and post-compose $R$ with $M\in\mathrm{PSL}_2(\C)$ such that $\pmb{c}=0, F(0)=\infty,$ and $0\in\mathfrak{D}$ with $R(0)=0$. With these normalizations, Relation~\ref{alg_formula} implies that $R$ maps $\infty$ to $\infty$ with local degree $2n$, and hence $R$ is a degree $2n$ polynomial.

The ideal boundary points of $\cH_\Gamma$ which are not fixed by $A_\Gamma$ form $(n-1)$ two-cycles. Hence, $F$ has $(n-1)$ two-cycles on $\partial\cD$, such that $F$ does not extend analytically to neighborhoods of these points (see Figure~\ref{hecke_punc_sphere_conf_mating_fig}).
The above observation and Relation~\ref{alg_formula} imply that there exist $c_1,\cdots,c_{n-1}\in\partial\mathfrak{D}$ such that for $j\in\{1,\cdots, n-1\}$, the points $c_j, \eta(c_j)$ are critical points of $R$, and they map under $R$ to these $(n-1)$ two-cycles of $F$.

We denote by $\pmb{x}_+$ the image of the fixed point of $A_\Gamma$ on $\partial\cH_\Gamma\setminus\{1\}$ under $\mathfrak{X}_\Gamma$ (see Figure~\ref{hecke_punc_sphere_conf_mating_fig}). By construction, $F(\pmb{x}_+)=\pmb{x}_+$, and $\pmb{x}_+$ is not a cut-point of $\partial\mathcal{D}$. Thus, the unique $R$-preimage of $\pmb{x}_+$ on $\partial\mathfrak{D}$ is a fixed point of $\eta$. After possibly pre-composing $R$ with $z\mapsto -z$, we can assume that $R(1)=\pmb{x}_+$. Moreover, the fact that $F$ does not extend analytically to a neighborhood of $\pmb{x}_+$ (because $A_\Gamma$ does not extend analytically to a relative neighborhood in $\overline{\D}$ of the ideal boundary points of $\cH_\Gamma$) implies that $1$ is a critical point of $R$; i.e., $R'(1)=0$. Thus, the finite critical points of $R$ are of the form
$$
\Big\{1,c_1, c_1^{-1},\cdots,c_{n-1},c_{n-1}^{-1}\Big\}.
$$
After possibly post-composing $R$ with a scaling, we have that 
$$
R(z)=z^{2n}+\sum_{j=1}^{2n}a_{2n-j}z^{2n-j}.
$$
As $R(0)=0$, we have $a_0=0$. The critical points of $R$ are the solutions of 
$$
R'(z)=2nz^{2n-1}+\sum_{j=1}^{2n-1}(2n-j)a_{2n-j}z^{2n-j-1}=0.
$$
By Vieta's formulas and the form of the critical points of $R$ given above, we have that
$$
a_1=-2n,\quad \textrm{and}\quad a_{2n-j}=-\frac{j+1}{2n-j}a_{j+1},\ j\in\{1,\cdots, n-1\}.
$$
Hence,
\begin{equation}
R(z)=z^{2n}-\sum_{j=1}^{n-1}\frac{j+1}{2n-j}a_{j+1} z^{2n-j}+\sum_{j=n}^{2n-2}a_{2n-j}z^{2n-j}-2nz.    
\label{R_reduced_form}
\end{equation}
Thus, $R$ depends only on the coefficients $a_2,\cdots,a_{n}$. 

We also note that the unique cut-point $\pmb{x}$ of $\partial\mathcal{D}$ is also a fixed point of $F$. Since $\partial\cD$ is topologically the wedge of two circles, there exist $\beta,\beta'\in\partial\mathfrak{D}$ ($\beta\neq \beta'$) such that $R(\beta)=R(\beta')=\pmb{x}$ (see Figure~\ref{hecke_punc_sphere_conf_mating_fig}). By Relation~\ref{alg_formula}, we have $\beta'=\eta(\beta)$. Further, the parabolic behavior of $A_1, A_\Gamma$ at $1$ translates to the fact that the two branches of $F$ at $\pmb{x}$ extend locally as tangent-to-identity parabolic germs. Hence, the complex number $\beta$ satisfies the equations 
$$
R(z)= R(\eta(z)),\quad \mathrm{and}\quad R'(\eta(z))\cdot \eta'(z)=R'(z).
$$
Note that the above equations are also satisfied by $z=1$. Thus, we have that $\deg_z\left(\gcd\left( R(z)- R(\eta(z)), R'(\eta(z))\cdot \eta'(z)-R'(z)\right)\right)\geq 2$. Hence the coefficients of $R$ satisfy the equation
$$
\mathrm{sRes}_1(R(z)- R(\eta(z)), R'(\eta(z))\cdot \eta'(z)-R'(z))=0,
$$
where $\mathrm{sRes}_j(P_1, P_2)$ denotes the $j$-th principal subresultant coefficient of two univariate polynomials $P_1,P_2\in\C[a_2,\cdots,a_n][z]$ (cf. \cite[\S 4.2.2]{BPC06}).

Therefore, the rational map $R$ lies on the $(n-2)$-dimensional algebraic variety
\begin{align}\label{alg_var_eqn}
\mathscr{V}:= \Big\{& \left(a_2,\cdots,a_{n}\right)\in\C^{n-1}: \\
\notag &\mathrm{sRes}_1\left(R(z)- R(\eta(z)), R'(\eta(z))\cdot \eta'(z)-R'(z)\right)=0\Big\}, 
\end{align}
where $R$ is given by Formula~\eqref{R_reduced_form}.

\subsection{Recovering marked groups from correspondences}\label{injective_subsec}

In what follows, we will denote the rational map $R$ associated with a marked group $\Gamma\in\mathrm{Teich}(S_{n+1})$ (constructed and normalized in Subsection~\ref{rat_norm_subsec}) by $R_\Gamma$.

We will explicate how the group $\Gamma$ and its preferred generating set can be recovered from the rational map $R_\Gamma$. Recall that the correspondence $\mathfrak{C}_\Gamma$ on $\widehat{\C}$ is defined as:
\begin{equation}
\{(\mathfrak{u}_1,\mathfrak{u}_2)\in\mathfrak{C}_\Gamma: \frac{R_\Gamma(\mathfrak{u}_2)-R_\Gamma(\eta(\mathfrak{u}_1))}{\mathfrak{u}_2-\eta(\mathfrak{u}_1)}=0\}.
\label{corr_eqn}
\end{equation}
The following result, which is an immediate consequence of Formulas~\ref{alg_formula} and~\eqref{corr_eqn}, underscores the role of $R_\Gamma$ as a mediator between the $F$-plane (where $F$ is the conformal mating between $A_1$ and $A_\Gamma$) and the $\mathfrak{C}_\Gamma$-plane.

\begin{prop}\label{lift_prop}
\noindent\begin{itemize}
\item Let $\mathfrak{u}_1\in\overline{\mathfrak{D}}$. Then, 
\begin{equation*}
(\mathfrak{u}_1,\mathfrak{u}_2)\in\mathfrak{C}_\Gamma \iff R_\Gamma(\mathfrak{u}_2)=F(R_\Gamma(\mathfrak{u}_1)),\ \mathfrak{u_2}\neq \eta(\mathfrak{u}_1).
\end{equation*}

\item Let $\mathfrak{u}_1\in\widehat{\C}\setminus\overline{\mathfrak{D}}$. Then, 
\begin{equation*}
(\mathfrak{u}_1,\mathfrak{u}_2)\in\mathfrak{C}_\Gamma \implies F(R_\Gamma(\mathfrak{u}_2))=R_\Gamma(\mathfrak{u}_1),\ \mathfrak{u_2}\neq \eta(\mathfrak{u}_1).
\end{equation*}
\end{itemize}
\end{prop}

By our normalization, $R_\Gamma(1)$ is a fixed point of $F$ on the limit set $\Lambda$. Since the iterated $F$-pre-images of this fixed point are dense on $\Lambda$, it follows by Proposition~\ref{lift_prop} that the grand orbit of $1$ under the correspondence $\mathfrak{C}_\Gamma$ is dense on $\widetilde{\Lambda}:=R_\Gamma^{-1}(\Lambda)$. Thus, the limit set $\widetilde{\Lambda}$ of $\mathfrak{C}_\Gamma$ can be recognized without referring to the conformal mating of $A_1$ and $A_\Gamma$. The limit set $\Lambda$ of the conformal mating and the limit set $\widetilde{\Lambda}$ of the correspondences are shown in black in  Figure~\ref{hecke_punc_sphere_conf_mating_fig}. 

We now look at the regular set $\Omega(\mathfrak{C}_\Gamma)=\widehat{\C}\setminus\widetilde{\Lambda}$. It consists of the sets 
$$
\mathcal{T}_\Gamma:=R_\Gamma^{-1}(\mathfrak{X}_\Gamma(\D))\quad \mathrm{and}\quad \mathcal{T}_1:=R_\Gamma^{-1}(\mathfrak{X}_1(\D)).
$$
Relation~\ref{alg_formula} and the fact that $\mathfrak{X}_\Gamma(\D), \mathfrak{X}_1(\D)$ are completely invariant under $F$ together imply that $\widetilde{\Lambda}, \mathcal{T}_\Gamma, \mathcal{T}_1$ are $\eta$-invariant.
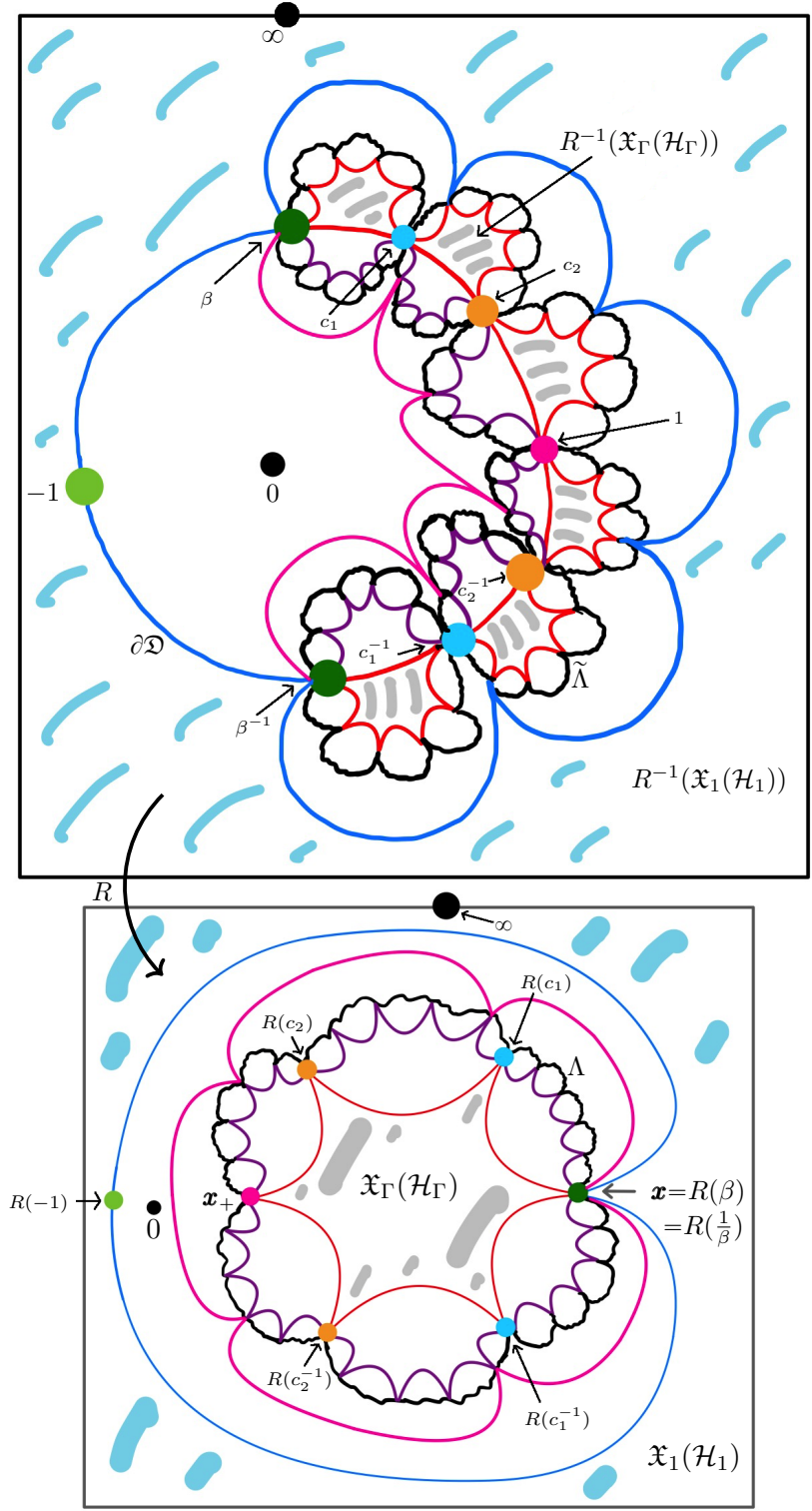
\begin{figure}[h!]
\captionsetup{width=0.9\linewidth}
\begin{tikzpicture}
\node[anchor=south west,inner sep=0] at (0,8.4) {\includegraphics[width=0.99\textwidth]{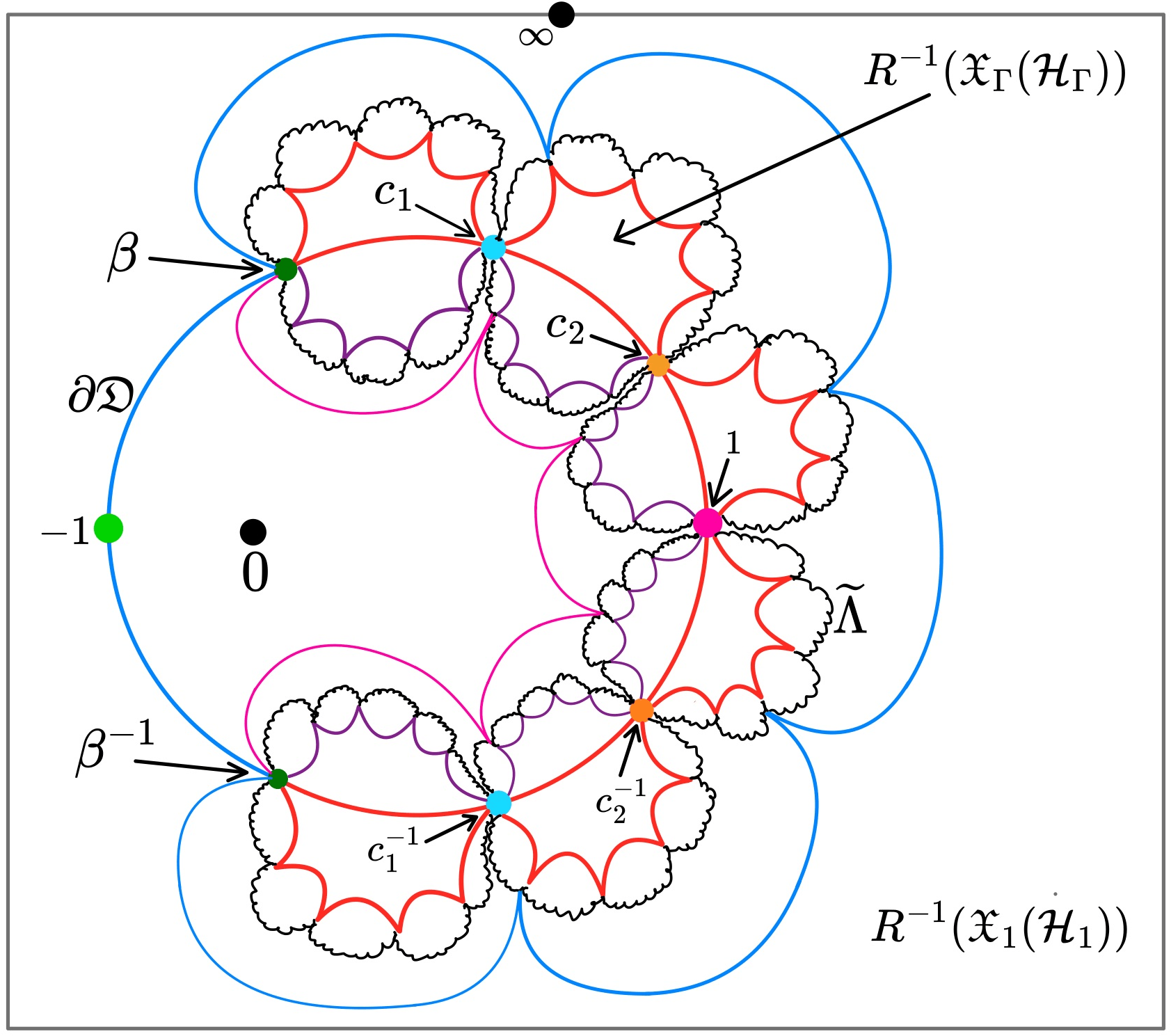}};
\node[anchor=south west,inner sep=0] at (0.8,-0.84) {\includegraphics[width=0.832\textwidth]{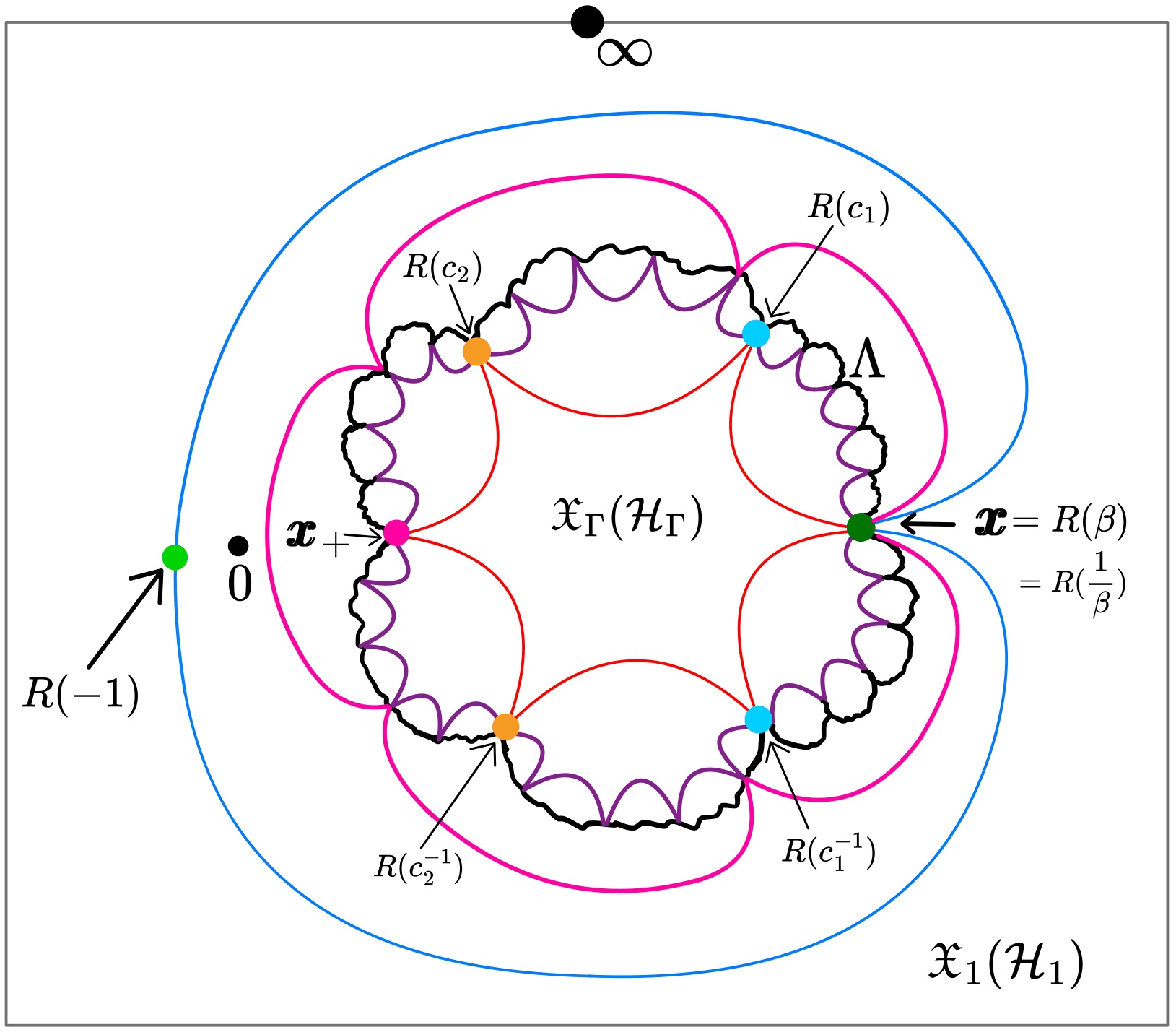}};
%\node at (5.25,4.4) {$\mathfrak{X}_\Gamma(\cH_\Gamma)$};
%\node at (9,0.8) {$\mathfrak{X}_1(\cH_1)$};
%\node at (9.06,4.32) {\begin{small}$\pmb{x}$$=$$R(\beta)$\end{small}};
%\node at (9.12,3.84) {\begin{small}$=$$R(\frac{1}{\beta})$\end{small}};
%\node at (1.88,3.84) {$0$};
%\node at (2.75,4.2) {\begin{small}$\pmb{x}_+$\end{small}};
%\node at (3.8,1.84) {\begin{tiny}$R(c_2^{-1})$\end{tiny}};
%\node at (3.64,6.56) {\begin{tiny}$R(c_2)$\end{tiny}};
%\node at (0.36,4.18) {\begin{tiny}$R(-1)$\end{tiny}};
%\node at (7.2,1.32) {\begin{tiny}$R(c_1^{-1})$\end{tiny}};
%\node at (7.05,7.1) {\begin{tiny}$R(c_1)$\end{tiny}};
%\node at (6.49,7.86) {\begin{tiny}$\infty$\end{tiny}};
%\node at (9.2,9.8) {\begin{small}$R^{-1}(\mathfrak{X}_1(\cH_1))$\end{small}};
%\node at (8.28,18.25) {\begin{small}$R^{-1}(\mathfrak{X}_\Gamma(\cH_\Gamma))$\end{small}};
\draw [->, line width=0.5mm] (1.88,9.6) to [out=225,in=135] (1.96,7.2);
\node at (1.12,8.84) {\begin{Large}$R$\end{Large}};
%\node at (3.45,19.66) {\begin{small}$\infty$\end{small}};
%\node at (3.45,13.6) {\begin{small}$0$\end{small}};
%\node at (0.42,13.6) {\begin{small}$-1$\end{small}};
%\node at (4.2,15.84) {\begin{tiny}$c_1$\end{tiny}};
%\node at (4.8,11.5) {\begin{tiny}$c_1^{-1}$\end{tiny}};
%\node at (7.4,16.6) {\begin{tiny}$c_2$\end{tiny}};
%\node at (6.1,12.32) {\begin{tiny}$c_2^{-1}$\end{tiny}};
%\node at (8.8,14.6) {\begin{tiny}$1$\end{tiny}};
%\node at (2.58,16.16) {\begin{tiny}$\beta$\end{tiny}};
%\node at (3.2,10.5) {\begin{tiny}$\beta^{-1}$\end{tiny}};
%\node at (1.8,11.6) {\begin{small}$\partial\mathfrak{D}$\end{small}};
%\node at (7.45,6) {\begin{small}$\Lambda$\end{small}};
\node at (7.5,11.2) {\begin{small}$\widetilde{\Lambda}$\end{small}};
\end{tikzpicture}
\caption{Illustrated are the correspondence plane (top) and the conformal mating plane (bottom) for $n=3$.}
\label{hecke_punc_sphere_conf_mating_fig}
\end{figure}

Since $\infty\in \mathfrak{X}_1(\cH_1)$ and $R_\Gamma$ is a polynomial, it follows that $\mathcal{T}_1$ is a simply connected domain. In particular, $\mathcal{T}_1$ is the unique component of $\widehat{\C}\setminus\widetilde{\Lambda}$ containing a critical value of $R_\Gamma$.
Further, the action of $\mathfrak{C}_\Gamma$ on $\mathcal{T}_1$ is generated by $\eta$ and the $2n-1$ non-trivial deck transformations of the branched covering $R_\Gamma:\mathcal{T}_1\to\mathfrak{X}_1(\D)$ (which is fully branched over $\infty$ and is unbranched otherwise).
By the proof of \cite[Proposition~5.13]{MM2}, these conformal automorphisms of $\mathcal{T}_1$ generate a group that acts properly discontinuously on $\mathcal{T}_1$, and the corresponding quotient is biholomorphic to the Hecke orbifold~$\Sigma_1$.

On the other hand, since $R_\Gamma$ has no critical value in $\mathfrak{X}_\Gamma(\D)$, it follows that $\mathcal{T}_\Gamma$ is the union of $2n$ simply connected domains each of which maps conformally onto $\mathfrak{X}_\Gamma(\D)$ under $R_\Gamma$. Hence, the deck transformations of the covering map $R_\Gamma:\mathcal{T}_\Gamma\to\mathfrak{X}_\Gamma(\D)$ permute the $2n$ components of $\mathcal{T}_\Gamma$ transitively. As before, the action of $\mathfrak{C}_\Gamma$ on $\mathcal{T}_\Gamma$ is generated by $\eta$ and the above deck transformations. It is easy to see from the dynamical structure of the conformal mating plane and $\eta$-invariance of $\mathcal{T}_1$ that the components of $\mathcal{T}_1$ can be enumerated as $U_1,\cdots,U_{2n}$ satisfying the following properties.
\begin{enumerate}
\item $\eta(U_j)=U_{2n+1-j}$, $j\in\{1,\cdots,n\}$.
\item $\partial U_1$ (respectively, $\partial U_{2n}$) touches $\partial U_2$ (respectively, $\partial U_{2n-1}$) only.
\item $\partial U_j$ touches $\partial U_{j-1}$ and $\partial U_{j+1}$ only, for $j\in\{2,\cdots,2n-1\}$.
\item The finite critical points of $R_\Gamma$ are the points of intersections of various $\partial U_j$s.
\end{enumerate}
(See Figure~\ref{hecke_punc_sphere_conf_mating_fig}.)
Further, Relation~\ref{alg_formula} and the fact that the conformal mating $F$ of $A_1$ and $A_\Gamma$ has exactly two parabolic fixed points on its limit set $\Lambda$ imply that there exists a unique $\beta\in\partial U_1$ such that
$$
R_\Gamma(\beta)= R_\Gamma(\beta^{-1}),\quad \mathrm{and}\quad R'(\eta(\beta))\cdot \eta'(\beta)=R'(\beta).
$$

Finally, construct a Jordan curve $\mathfrak{J}$ by connecting the finite critical points of $R_\Gamma$ and $\beta, \beta^{-1}$ by hyperbolic geodesics in the simply connected domains $U_1,\cdots,U_{2n}, \mathcal{T}_1$. By construction, $\eta(\mathfrak{J})=\mathfrak{J}$. It now follows from the relation between the conformal mating and correspondence planes and the normalization of $R_\Gamma$ that the map $R_\Gamma$ is univalent on one of the complementary components of $\mathfrak{J}$ (in this case, it is the component of $\widehat{\C}\setminus\mathfrak{J}$ containing the origin), and this component coincides with the domain $\mathfrak{D}$ such that $\overline{R_\Gamma(\mathfrak{D})}$ is the domain of definition of the conformal mating $F$ of $A_1$ and $A_\Gamma$. Thus, $R_\Gamma$ completely determines the conformal mating of $A_1$ and $A_\Gamma$. In particular, the Bowen-Series map of $\Gamma$ can be recovered from $R_\Gamma$. Since the Bowen-Series map of $\Gamma$ encodes a preferred generating set of $\Gamma$, we can recover the group $\Gamma$ and its preferred generating set from $R_\Gamma$.

As a consequence of the preceding discussion, we have the following result.

\begin{prop}\label{injective_prop}
The map 
\begin{align*}
\pmb{\Psi}:\ & \mathrm{Teich}(S_{0,n+1})\to\mathscr{V}\subset\C^{n-1}\\
&\hspace{2cm} \Gamma\mapsto R_\Gamma
\end{align*}
is injective. Here, $\mathscr{V}$ is the $(n-2)-$dimensional algebraic variety defined by Equation~\eqref{alg_var_eqn}.
\end{prop}

\noindent\textbf{Figure~\ref{hecke_punc_sphere_conf_mating_fig}:} The conformal mating $F$ is defined on the `pinched annulus' $\widehat{\C}\setminus\Int{\left(\mathfrak{X}_\Gamma(\cH_\Gamma)\cup\mathfrak{X}_1(\cH_1)\right)}$, where $\mathfrak{X}_\Gamma(\cH_\Gamma)$ is the bounded region enclosed by the red hexagon, and $\mathfrak{X}_1(\cH_1)$ is the region outside the blue monogon. The black fractal Jordan curve is the limit set $\Lambda$ of the conformal mating $F$. The bounded complementary component of $\Lambda$ is $\mathfrak{X}_\Gamma(\D)$, while the unbounded complementary component is $\mathfrak{X}_1(\D)$. The first preimages of $\mathfrak{X}_\Gamma(\cH_\Gamma)$, $\mathfrak{X}_1(\cH_1)$ under $F$ are also shown.

The set $\mathfrak{X}_1(\cH_1)$ (respectively, $\mathfrak{X}_1(\D)$) has a connected preimage under $R_\Gamma$ in the correspondence plane, and it maps onto $\mathfrak{X}_1(\cH_1)$ (respectively, $\mathfrak{X}_1(\D)$) as a $6:1$ branched cover ramified only at $\infty$. On the other hand, $R_\Gamma^{-1}(\mathfrak{X}_\Gamma(\cH_\Gamma))$ (respectively, $R_\Gamma^{-1}(\mathfrak{X}_\Gamma(\D))$) has six connected components, each of which maps conformally onto $\mathfrak{X}_\Gamma(\cH_\Gamma)$ (respectively, $\mathfrak{X}_\Gamma(\D)$). The limit set $\widetilde{\Lambda}=R_\Gamma^{-1}(\Lambda)$ of the correspondence $\mathfrak{C}_\Gamma$ is a chain of six Jordan curves that divides the sphere into the fully invariant sets $\mathcal{T}_\Gamma=R_\Gamma^{-1}(\mathfrak{X}_\Gamma(\D))\quad \mathrm{and}$ and $ \mathcal{T}_1=R_\Gamma^{-1}(\mathfrak{X}_1(\D))$.
\subsection{Holomorphic injection of $\mathrm{Teich}(S_{0,n+1})$ into a space of correspondences}\label{holo_subsection}

\begin{prop}\label{holomorphic_prop}
The map 
\begin{align*}
\pmb{\Psi}:\ & \mathrm{Teich}(S_{0,n+1})\to\mathscr{V}\subset\C^{n-1}\\
&\hspace{2cm} \Gamma\mapsto R_\Gamma
\end{align*}
is holomorphic, where $\mathrm{Teich}(S_{0,n+1})$ is identified with the Bers slice of the group $\Gamma_2$.
\end{prop}
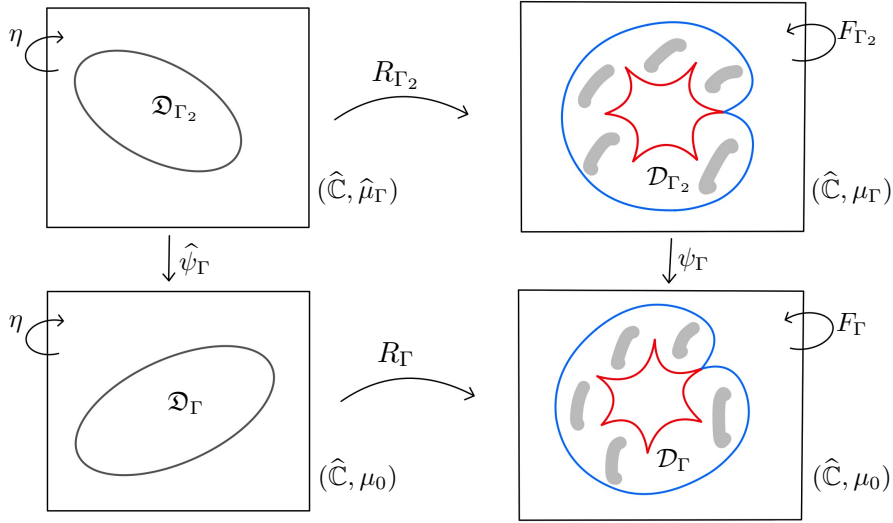
\begin{figure}[ht]
\captionsetup{width=0.96\linewidth}
\begin{tikzpicture}
\node[anchor=south west,inner sep=0] at (0,0) {\includegraphics[width=0.9\linewidth]{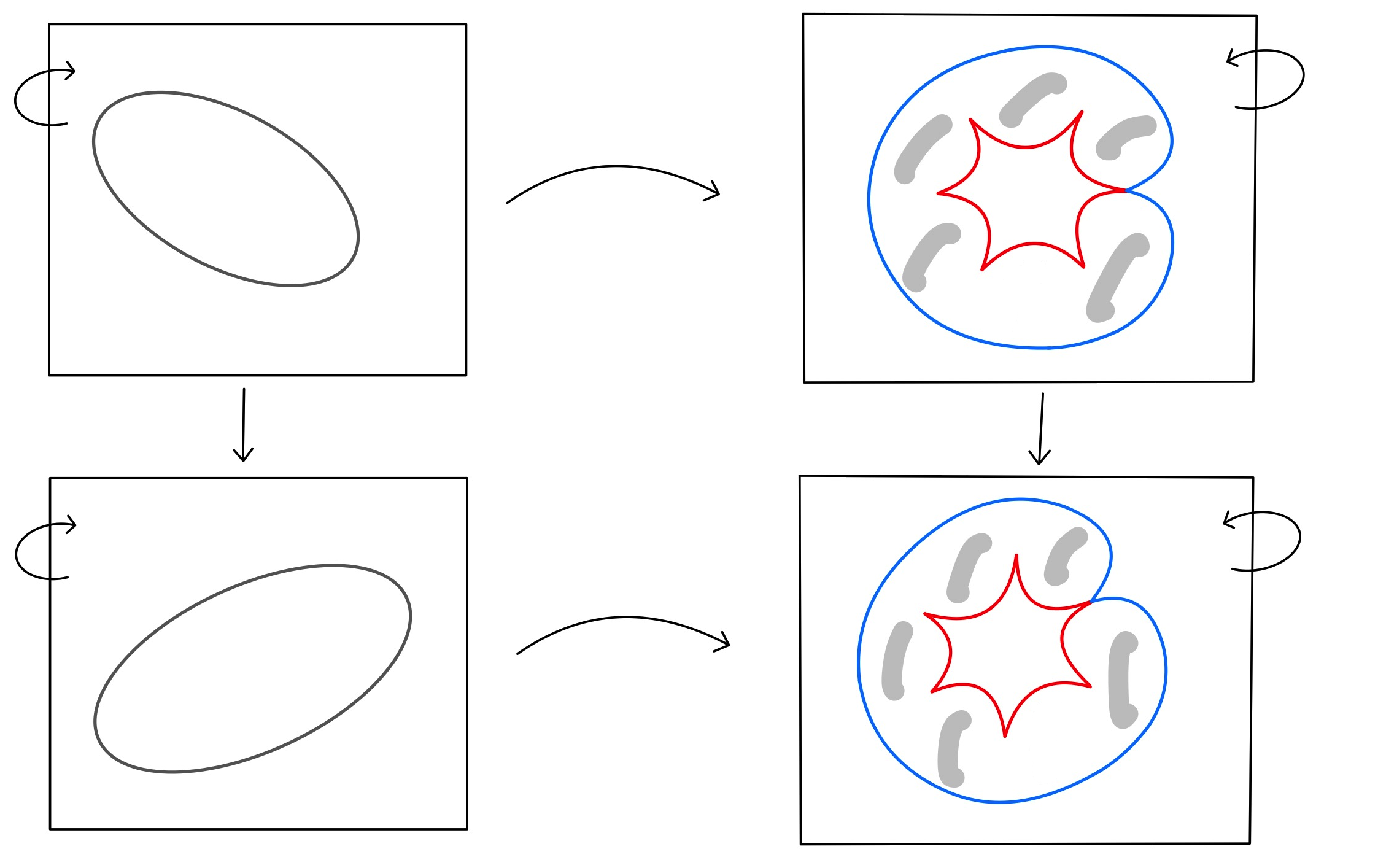}};
\node at (5,6.1) {$R_{\Gamma_2}$};
\node at (5,2.4) {$R_\Gamma$};
\node at (11.08,6.66) {\begin{small}$F_{\Gamma_2}$\end{small}};
\node at (11,2.8) {\begin{small}$F_{\Gamma}$\end{small}};
\node at (0,2.8) {\begin{small}$\eta$\end{small}};
\node at (0,6.6) {\begin{small}$\eta$\end{small}};
\node at (2.36,3.66) {\begin{small}$\widehat{\psi}_\Gamma$\end{small}};
\node at (8.92,3.66) {\begin{small}$\psi_\Gamma$\end{small}};
\node at (2.24,1.75) {\begin{small}$\mathfrak{D}_{\Gamma}$\end{small}};
\node at (2.1,5.66) {\begin{small}$\mathfrak{D}_{\Gamma_2}$\end{small}};
\node at (8.6,4.72) {\begin{small}$\cD_{\Gamma_2}$\end{small}};
\node at (8.66,1) {\begin{small}$\cD_{\Gamma}$\end{small}};
\node at (4.5,0.8) {\begin{small}$(\widehat{\C},\mu_0)$\end{small}};
\node at (4.5,4.6) {\begin{small}$(\widehat{\C},\widehat{\mu}_\Gamma)$\end{small}};
\node at (11,0.8) {\begin{small}$(\widehat{\C},\mu_0)$\end{small}};
\node at (11,4.6) {\begin{small}$(\widehat{\C},\mu_\Gamma)$\end{small}};
\end{tikzpicture}
\caption{The relation between the uniformizing rational maps $R_{\Gamma_2}$ and $R_\Gamma$ is shown.}
\label{qc_deform_fig}
\end{figure}
\begin{proof}
The Bers slice $\mathcal{B}(\Gamma_2)$ of $\Gamma_2$ consists of group isomorphisms $\rho:\Gamma_2\to\Gamma$ given by
$$
\rho(g)=\psi_\rho\circ g\circ \psi_\rho^{-1},\quad g\in\Gamma_2,
$$
where $\psi_\rho$ is a quasiconformal homeomorphism of $\widehat{\C}$ that is conformal on $\widehat{\C}\setminus\overline{\D}$ (cf. \cite[\S 5.10]{Marden}). The quasiconformal maps $\psi_\rho$ depend holomorphically on the complex coordinates on the Bers slice $\mathcal{B}(\Gamma_2)$. We denote the standard complex structure on the Riemann sphere by $\mu_0$, and set $\mu_\rho:=\psi_\rho^*(\mu_0)$. By construction, $\mu_\rho$ depends holomorphically on $\rho$, and is $\Gamma_2$-invariant. Note that the quasiconformal maps $\psi_\rho$ also conjugates $A_{\Gamma_2}$ to $A_\Gamma$.

We denote the normalized conformal mating of $A_1$ and $A_\Gamma$ by $F_\Gamma$, and the associated mating conjugacies by $\mathfrak{X}_1, \mathfrak{X}_\Gamma$. The Beltrami coefficient $\mu_\rho$ can be pushed forward to the $F_{\Gamma_2}$-plane by the mating conjugacy $\mathfrak{X}_{\Gamma_2}$ to yield
$$
\mu_{\Gamma}:=\begin{cases}
\left(\mathfrak{X}_{\Gamma_2}\right)_*(\mu_\rho)\quad \textrm{on}\quad \mathfrak{X}_{\Gamma_2}(\D),\\
\quad 0\hspace{2cm} \textrm{elsewhere}.
\end{cases}
$$ 
Clearly, $\mu_\Gamma$ is $F_{\Gamma_2}$-invariant, and depends holomorphically on the marked group $\Gamma$. Consequently, the quasiconformal homeomorphisms $\psi_\Gamma$ solving the Beltrami equation with coefficient $\mu_\Gamma$ depend holomorphically on $\Gamma$. Further, if we normalize $\psi_\Gamma$ appropriately, then $\psi_\Gamma\circ F_{\Gamma_2}\circ \psi_\Gamma^{-1}$ is the normalized conformal mating $F_\Gamma$ of $A_1$ and $A_\Gamma$ with mating conjugacies $\psi_\Gamma\circ\mathfrak{X}_1$ and $\psi_\Gamma\circ\mathfrak{X}_{\Gamma_2}\circ \psi_\rho^{-1}$.
Hence, the normalized conformal matings $F_{\Gamma_2}:\overline{\cD_{\Gamma_2}}\to\widehat{\C}$ (of $A_1$ and $A_{\Gamma_2}$) and $F_\Gamma:\overline{\cD_\Gamma}\to\widehat{\C}$ (of $A_1$ and $A_\Gamma$) are quasiconformally conjugate by a global quasiconformal homeomorphism $\psi_\Gamma$ that depends holomorphically on $\Gamma$. 

Let $R_{\Gamma}$ be as in Subsection~\ref{injective_subsec} for a marked group $\Gamma\in\mathrm{Teich}(\Gamma_2)$. We also denote by $\mathfrak{D}_\Gamma$ the Jordan domain satisfying Conditions~\ref{eta_inv_cond},~\ref{univ_cond}, and~\ref{alg_formula}.  
We define $\widehat{\mu}_\Gamma:=R_{\Gamma_2}^*(\mu_\Gamma)$, and note that $\widehat{\mu}_\Gamma$ also depends holomorphically on $\Gamma$.
The relation $F_{\Gamma_2}\circ R_{\Gamma_2}\equiv R_{\Gamma_2}\circ\eta$ (on $\mathfrak{D}_{\Gamma_2}$) and $F_{\Gamma_2}$-invariance of $\mu_\Gamma$ imply that $\widehat{\mu}_\Gamma$ is an $\eta$-invariant Beltrami coefficient. 
Let $\widehat{\psi}_\Gamma$ be a quasiconformal homeomorphism of the sphere solving the Beltrami equation with coefficient $\widehat{\mu}_\Gamma$; i.e., $\widehat{\psi}_\Gamma^*(\mu_0)=\widehat{\mu}_\Gamma$. Then, $\widehat{\psi}_\Gamma\circ\eta\circ\widehat{\psi}_\Gamma^{-1}$ is a M{\"o}bius involution. We normalize $\widehat{\psi}_\Gamma$ so that it sends $\pm 1, \infty$ to $\pm 1,\infty$ (respectively). It then follows that $\widehat{\psi}_\Gamma$ depends holomorphically on $\Gamma$, and conjugates $\eta$ to itself. It is now easy to see that $\psi_\Gamma\circ R_{\Gamma_2}\circ\widehat{\psi}_\Gamma^{-1}$ is a quasiregular map of $\widehat{\C}$ preserving the standard complex structure, and hence is a rational map (see Figure~\ref{qc_deform_fig}). 
Further, the rational map $\psi_\Gamma\circ R_{\Gamma_2}\circ\widehat{\psi}_\Gamma^{-1}$ is injective on $\widehat{\psi}_\Gamma(\mathfrak{D}_{\Gamma_2})$, and we have
$$
F_\Gamma\circ\left(\psi_\Gamma\circ R_{\Gamma_2}\circ\widehat{\psi}_\Gamma^{-1}\right)\equiv \left(\psi_\Gamma\circ R_{\Gamma_2}\circ\widehat{\psi}_\Gamma^{-1}\right)\circ\eta
$$
on $\widehat{\psi}_\Gamma(\mathfrak{D}_{\Gamma_2})$. By the uniqueness statement~\ref{pre_comp} and our normalization of $\widehat{\psi}_\Gamma$, we now have that
$$
\mathfrak{D}_\Gamma=\widehat{\psi}_\Gamma(\mathfrak{D}_{\Gamma_2}),\quad \mathrm{and}\quad R_\Gamma=\psi_\Gamma\circ R_{\Gamma_2}\circ\widehat{\psi}_\Gamma^{-1}.
$$
Thanks to the holomorphic dependence of the quasiconformal homeomorphisms $\psi_\Gamma$ and $\widehat{\psi}_\Gamma$ on $\Gamma$, the rational map $R_\Gamma$ (more precisely, the coefficients of $R_\Gamma$) depend holomorphically as the marked group $\Gamma$ runs over the Bers slice $\mathcal{B}(\Gamma_2)$.  
\end{proof}

\begin{remark}
The group $\Gamma_2$ merely plays the role of a base-point in the proof of Proposition~\ref{holomorphic_prop}; while it is used to put a complex structure on $\mathrm{Teich}(S_{0,n+1})$, the image of the map $\pmb{\Psi}$ does not depend on $\Gamma_2$. In fact, for any $\Gamma_2'\in\mathrm{Teich}(S_{0,n+1})$, the corresponding Bers slices $\mathcal{B}(\Gamma_2)$ and $\mathcal{B}(\Gamma_2')$ are naturally biholormorphic, and hence one can use any such Bers slice $\mathcal{B}(\Gamma_2')$ to equip $\mathrm{Teich}(S_{0,n+1})$ with a complex structure.
\end{remark}

Combining Propositions~\ref{injective_prop} and~\ref{holomorphic_prop}, we readily obtain the following:

\begin{theorem}\label{bers_embedding_thm}
Let $n\geq 2$. There exists an injective holomorphic map 
\begin{align*}
\pmb{\Psi}:\ & \mathrm{Teich}(S_{0,n+1})\to\mathscr{V}\subset\C^{n-1}\\
&\hspace{2cm} \Gamma\mapsto R_\Gamma,
\end{align*}
where $\mathscr{V}$ is the $(n-2)-$dimensional algebraic variety defined by Equation~\eqref{alg_var_eqn}, such that the correspondence $\mathfrak{C}$ defined by Equation~\eqref{corr_eqn} simultaneously uniformizes the (marked) $(n+1)-$times punctured sphere $\D/\Gamma$ and the Hecke orbifold (having one puncture, an order 2 orbifold point and an order $2n$ orbifold point).
\end{theorem}

\noindent We conclude with several questions and comments.

\subsection*{Non-singularity and holomorphic embedding of $\mathrm{Teich}(S_{0,n+1})$ into a space of correspondences}

If $\pmb{\Psi}(\mathrm{Teich}(S_{0,n+1}))$ contains no singular point of $\mathscr{V}$, then it is a complex manifold of dimension $n-2$, and hence the injective holomorphic map $\pmb{\Psi}$ would be a biholomorphism onto its image. To this end, we ask:

\begin{question}
Is the variety $\mathscr{V}$ defined by Equation~\eqref{alg_var_eqn} non-singular? If it is not, does the image of $\pmb{\Psi}$ contains a singular point of $\mathscr{V}$?
\end{question}

\subsection*{Boundedness of Bers slices} 
Classical Bers slices are pre-compact in the $\PSL_2(\C)$-character variety \cite{Ber70}. As mentioned the introduction, holomorphic embeddings of Bers slices of genus zero orbifolds in spaces of algebraic correspondences was constructed in \cite{MM2}, with the property that the correspondences in the image are matings of the corresponding surfaces and the polynomial $z^d$. Pre-compactness of these images was proved in \cite{LMM25}. It is natural to ask the same question in the current setting.

\begin{question}
1) Is the image of the map $\pmb{\Psi}$ bounded in the variety $\mathscr{V}$?  

\noindent 2) Analyze the dynamics of the correspondences lying on the boundary of the image of $\pmb{\Psi}$.
\end{question}

Furthermore, many of the questions posed in \cite[Question~1.2]{LMM25} make sense in the present setting too.

\subsection*{Higher Bowen-Series maps, and algebraic correspondences}
The arguments of Proposition~\ref{conf_mating_exists_prop} can also be used to construct conformal matings of higher Bowen-Series maps and factor Bowen-Series maps of the same degree. However, we do not know if such matings admit algebraic descriptions analogous to the one given in Proposition~\ref{conf_mating_b_inv_prop} (which asserts that the mating of two factor Bowen-Series maps is an algebraic function). We raise the following question:
\begin{question}
  Can the conformal matings of higher Bowen-Series maps and factor Bowen-Series maps (of equal degree) be lifted to produce algebraic correspondences uniformizing a pair of genus zero orbifolds?
\end{question}

\subsection*{Quasi-Fuchsian locus, and boundary correspondences}
In Theorem~\ref{simult_unif_corr_thm}, we consrtucted correspondences that combine genus zero $\Sigma_1, \Sigma_2$ satisfying certain compatibility conditions. While we focussed on the special case where $\Sigma_1$ is the Hecke orbifold in Section~\ref{bers_sec}, one can consider the situation where both $\Sigma_1$ and $\Sigma_2$ have non-trivial Teichm{\"u}ller spaces. In this general setup, the construction of this section would give rise to a holomorphic injection from $\mathrm{Teich}(\Sigma_1)\times\mathrm{Teich}(\Sigma_2)$ into an appropriate space of algebraic correspondences that combine marked surfaces in $\mathrm{Teich}(\Sigma_1)$, $\mathrm{Teich}(\Sigma_2)$. The following problem is of considerable interest.

\begin{question}
1) Prove an analog of Thurston's Double Limit Theorem for correspondences lying in the image of $\mathrm{Teich}(\Sigma_1)\times\mathrm{Teich}(\Sigma_2)$ under the above holomorphic injection.

2) Are there correspondences on the boundary of the image of $\mathrm{Teich}(\Sigma_1)\times\mathrm{Teich}(\Sigma_2)$ (under the above holomorphic injection) that are matings of two singly degenerate groups lying on $\partial \mathrm{Teich}(\Sigma_1)$, $\partial \mathrm{Teich}(\Sigma_2)$?
\end{question}

\end{document}